\documentclass[12pt]{article}

\usepackage[margin=1in]{geometry}
\usepackage{enumerate}
\usepackage{amssymb}
\usepackage{mathrsfs}
\usepackage{graphicx}
\usepackage{amsthm}
\usepackage{amsmath}
\usepackage{bbm}
\usepackage{tikz}

\newtheorem{definition}{Definition}[section]
\newtheorem{theorem}[definition]{Theorem}
\newtheorem{lemma}[definition]{Lemma}
\newtheorem{corollary}[definition]{Corollary}

\newtheorem{example}[definition]{Example}

\newtheorem{proposition}[definition]{Proposition}

\def\makeCal#1{%
\expandafter\newcommand\csname c#1\endcsname{\mathcal{#1}}}
\def\makeBf#1{%
\expandafter\newcommand\csname b#1\endcsname{\mathbf{#1}}}
\def\makeBb#1{%
\expandafter\newcommand\csname m#1\endcsname{\mathbb{#1}}}
\def\makeFrak#1{%
\expandafter\newcommand\csname f#1\endcsname{\mathfrak{#1}}}
\def\makeScr#1{%
\expandafter\newcommand\csname s#1\endcsname{\mathscr{#1}}}

\makeatletter
\count@=0
\loop
\advance\count@ 1
\edef\y{\@Alph\count@}%
\expandafter\makeCal\y
\expandafter\makeBf\y
\expandafter\makeBb\y
\expandafter\makeFrak\y
\expandafter\makeScr\y
\ifnum\count@<26
\repeat
\makeatother

\DeclareMathOperator{\Cat}{Cat}
\DeclareMathOperator{\m1}{\mathbbm{1}}
\newcommand{\tG}{\tilde{G}}

\usepackage{xypic}

\newcommand{\qpermute}[2]{\genfrac{\{}{\}}{0pt}{}{#1}{#2}_q}

\title{A uniform approach to the Damiani, Beck, and alternating PBW bases for the positive part of $U_q(\widehat{\mathfrak{sl}}_2)$}
\author{Chenwei Ruan}
\date{}

\begin{document}
\maketitle

\begin{abstract}
This paper is about the positive part $U_q^+$ of the $q$-deformed enveloping algebra $U_q(\widehat{\mathfrak{sl}}_2)$. The literature contains at least three PBW bases for $U_q^+$, called the Damiani, the Beck, and the alternating PBW bases. These PBW bases are related via exponential formulas. In this paper, we introduce an exponential generating function whose argument is a power series involving the Beck PBW basis and an integer parameter $m$. The cases $m=2$ and $m=-1$ yield the known exponential formulas for the Damiani and alternating PBW bases, respectively. The case $m=1$ appears in the author's previous paper. In the present paper, we give a comprehensive study of the generating function for an arbitrary integer $m$. We have two main results. The first main result gives a factorization of the generating function. In the second main result, we express the coefficients of the generating function in closed form. 

\end{abstract}

\section{Introduction}
This paper involves the $q$-deformed enveloping algebra $U_q(\widehat{\mathfrak{sl}}_2)$ \cite{CP}. The algebra $U_q(\widehat{\mathfrak{sl}}_2)$ appears in the topics of combinatorics \cite{IT,JM,JKKKY,LIW}, quantum algebras \cite{baseilhac2,BG,drinfeld,FHZ,jimbo,watanabe}, and representation theory \cite{ariki,bittmann,FR,jing,XXZ}. We are interested in the positive part $U_q^+$ of $U_q(\widehat{\mathfrak{sl}}_2)$ \cite{CP,lusztig}. The algebra $U_q^+$ is associative, noncommutative, and infinite-dimensional. It has a presentation with two generators and two relations called the $q$-Serre relations; see \eqref{eq:qSerre1}, \eqref{eq:qSerre2} below. 

\medskip
\noindent In \cite{damiani}, Damiani obtained a Poincar\'e-Birkhoff-Witt (or PBW) basis for $U_q^+$. The basis elements $\{E_{n\delta+\alpha_0}\}_{n=0}^\infty$, $\{E_{n\delta+\alpha_1}\}_{n=0}^\infty$, $\{E_{n\delta}\}_{n=1}^\infty$ were defined recursively using a braid group action. 

\medskip
\noindent In \cite{beck}, Beck obtained a PBW basis for $U_q^+$ from the Damiani PBW basis, by replacing $E_{n\delta}$ with an element $E^{\rm Beck}_{n\delta}$ for $n \geq 1$. In \cite{BCP}, Beck, Chari, and Pressley showed that the elements $\{E_{n\delta}\}_{n=1}^\infty$ and $\{E^{\rm Beck}_{n\delta}\}_{n=1}^\infty$ are related via an exponential formula. 

\medskip
\noindent In \cite{rosso1,rosso2}, Rosso introduced an embedding of $U_q^+$ into a $q$-shuffle algebra. In Section 2 we will review this algebra in detail, and for now we give a quick preview. Let $\mV$ denote the free associative algebra with two generators $x,y$. A free product of generators is called a word. The vector space $\mV$ has a basis consisting of the words. In \cite{rosso1,rosso2}, Rosso introduced a second algebra structure on $\mV$, called the $q$-shuffle algebra. He then gave an embedding of $U_q^+$ into the $q$-shuffle algebra $\mV$. 

\medskip
\noindent In \cite{ter_catalan}, Terwilliger used the Rosso embedding to obtain a closed form for the Damiani PBW basis elements. This closed form involves some words of a certain type said to be Catalan. Let $\overline{x}=1$ and $\overline{y}=-1$. A word $a_1a_2 \cdots a_n$ is Catalan whenever $\overline{a}_1+\overline{a}_2+\cdots+\overline{a}_i \geq 0$ for $1 \leq i \leq n-1$ and $\overline{a}_1+\overline{a}_2+\cdots+\overline{a}_n=0$. The length of a Catalan word is even. For $n \geq 0$, Terwilliger introduced the element $C_n$ in $\mV$. He defined $C_n$ as a linear combination of the Catalan words of length $2n$. For a Catalan word $a_1a_2 \cdots a_{2n}$, its coefficient in $C_n$ is 
\begin{equation*}
\prod_{i=1}^{2n}[1+\overline{a}_1+\overline{a}_2+\cdots+\overline{a}_i]_q. 
\end{equation*}

\noindent In \cite[Theorem 1.7]{ter_catalan}, Terwilliger showed that the Rosso embedding sends 
\begin{equation*}
E_{n\delta+\alpha_0} \mapsto q^{-2n}(q-q^{-1})^{2n}xC_n, \hspace{4em} E_{n\delta+\alpha_1} \mapsto q^{-2n}(q-q^{-1})^{2n}C_ny
\end{equation*}
for $n \geq 0$, and 
\begin{equation*}
E_{n\delta} \mapsto -q^{-2n}(q-q^{-1})^{2n-1}C_n
\end{equation*}
for $n \geq 1$. 

\medskip
\noindent In \cite{ter_beck}, Terwilliger used the Rosso embedding to obtain a closed form for the elements $\{E^{\rm Beck}_{n\delta}\}_{n=1}^\infty$. He introduced the elements $\{xC_ny\}_{n=0}^\infty$ in $\mV$. In \cite[Theorem 7.1]{ter_beck}, Terwilliger showed that the Rosso embedding sends 
\begin{equation*}
E^{\rm Beck}_{n\delta} \mapsto \frac{[2n]_q}{n}q^{-2n}(q-q^{-1})^{2n-1}xC_{n-1}y
\end{equation*}

\noindent for $n \geq 1$. 

\medskip
\noindent In \cite{ter_alternating}, Terwilliger introduced the alternating PBW basis for $U_q^+$. He introduced the words 
\begin{align*}
&W_0=x, \hspace{1em} W_{-1}=xyx, \hspace{1em} W_{-2}=xyxyx, \hspace{1em} W_{-3}=xyxyxyx, \hspace{1em} \ldots \\
&W_1=y, \hspace{1em} W_2=yxy, \hspace{1em} W_3=yxyxy, \hspace{1em} W_4=yxyxyxy, \hspace{1em} \ldots \\
&G_1=yx, \hspace{1em} G_2=yxyx, \hspace{1em} G_3=yxyxyx, \hspace{1em} G_4=yxyxyxyx, \hspace{1em} \ldots \\
&\tG_1=xy, \hspace{1em} \tG_2=xyxy, \hspace{1em} \tG_3=xyxyxy, \hspace{1em} \tG_4=xyxyxyxy, \hspace{1em} \ldots 
\end{align*}

\noindent In \cite[Theorem 10.1]{ter_alternating}, Terwilliger showed that the words $\{W_{-n}\}_{n=0}^\infty$, $\{W_{n+1}\}_{n=0}^\infty$, $\{\tG_n\}_{n=1}^\infty$ form a PBW basis for $U_q^+$. In \cite[Theorem 9.15]{ter_alternating}, he gave $\{G_n\}_{n=1}^\infty$ in terms of $\{W_{-n}\}_{n=0}^\infty$, $\{W_{n+1}\}_{n=0}^\infty$, $\{\tG_n\}_{n=1}^\infty$. 

\medskip
\noindent In order to illuminate the algebraic structure of the above PBW bases, we consider the generating functions 
\begin{equation*}
C(t)=\sum_{n=0}^\infty C_nt^n, \hspace{4em} \tG(t)=\sum_{n=0}^\infty \tG_nt^n, 
\end{equation*}
where $\tG_0=1$. 

\noindent In \cite[Section 9]{ter_alternating}, Terwilliger considered the multiplicative inverse of $\tG(t)$ with respect to the $q$-shuffle product; this inverse is denoted by $D(t)$. Following \cite[Definition 9.11]{ter_alternating} we write 
\begin{equation*}
D(t)=\sum_{n=0}^\infty D_nt^n. 
\end{equation*}

\noindent In \cite{inverse}, the present author obtained a closed form for the elements $\{D_n\}_{n=0}^\infty$. We showed that for $n \geq 0$, the element $D_n$ is a linear combination of the Catalan words of length $2n$. Moreoever, for a Catalan word $a_1a_2 \cdots a_{2n}$, its coefficient in $D_n$ is 
\begin{equation*}
(-1)^n\prod_{i=1}^{2n}[\overline{a}_1+\overline{a}_2+\cdots+\overline{a}_{i-1}+(\overline{a}_i+1)/2]_q. 
\end{equation*}

\noindent As mentioned earlier, the elements $\{E_{n\delta}\}_{n=1}^\infty$ and $\{E^{\rm Beck}_{n\delta}\}_{n=1}^\infty$ are related via an exponential formula. In \cite[Section 8]{ter_beck}, Terwilliger used the Rosso embedding to reformulate this exponential formula as follows: 
\begin{equation}\label{eq:Cexp}
C(t)=\exp\left(\sum_{n=1}^\infty \frac{[2n]_q}{n}xC_{n-1}yt^n\right). 
\end{equation}

\noindent In \cite[Section 9]{ter_beck} Terwilliger showed 
\begin{equation}\label{eq:Dexp}
D(t)=\exp\left(\sum_{n=1}^\infty \frac{(-1)^n[n]_q}{n}xC_{n-1}yt^n\right), 
\end{equation}
\begin{equation}\label{eq:tGexp}
\tG(t)=\exp\left(-\sum_{n=1}^\infty \frac{(-1)^n[n]_q}{n}xC_{n-1}yt^n\right). 
\end{equation}

\noindent To summarize \eqref{eq:Cexp}--\eqref{eq:tGexp} in a uniform way, for an integer $m$ we consider the generating function 
\begin{equation}\label{eq:expgenfun}
\exp\left(\sum_{n=1}^\infty \frac{[mn]_q}{n}xC_{n-1}yt^n\right). 
\end{equation}

\noindent Setting $m=2,m=1,m=-1$ in $\eqref{eq:expgenfun}$, we get $C(t)$, $D(-t)$, $\tG(-t)$ respectively. 

\medskip
\noindent Motivated by this observation, in the present paper we investigate the generating function \eqref{eq:expgenfun} for
an arbitrary integer $m$. We have two main results. The first main result gives a factorization of \eqref{eq:expgenfun}. The factors involve $D(t)$ if $m$ is positive, and $\tG(t)$ if $m$ is negative. In the second main result, we express \eqref{eq:expgenfun} explicitly as a linear combination of Catalan words, and we give the coefficients. As we will see, for a Catalan word $a_1a_2 \cdots a_{2n}$, its coefficient in \eqref{eq:expgenfun} is equal to $t^n$ times 
\begin{equation*}
\prod_{i=1}^{2n}[\overline{a}_1+\overline{a}_2+\cdots+\overline{a}_{i-1}+m(\overline{a}_i+1)/2]_q. 
\end{equation*}

\noindent We will prove the two main results from scratch, without citing earlier results from the literature. 

\medskip
\noindent We will also obtain a third result which may be of independent interest. We will reprove \eqref{eq:Cexp}--\eqref{eq:tGexp} from scratch, without citing earlier results from the literature. 

\medskip
\noindent In Section 2, we review the $q$-shuffle algebra and give a precise statement of our main results. The rest of this paper is devoted to the proof of our main results. An outline of the proof is given at the end of Section 2. 

\section{Statement of the main results}
In this section, we motivate and state our main results. 

\medskip
\noindent Recall the integers $\mZ=\{0,\pm 1,\pm 2,\ldots\}$ and the natural numbers $\mN=\{0,1,2,\ldots\}$. Let $\mF$ denote a field of characteristic zero. All algebras in this paper are associative, over $\mF$, and have a multiplicative identity. Let $q$ denote a nonzero scalar in $\mF$ that is not a root of unity. For $n \in \mZ$, define 
\begin{equation}\label{eq:qint}
[n]_q=\frac{q^n-q^{-n}}{q-q^{-1}}. 
\end{equation}

\noindent Next we recall the positive part $U_q^+$ of the $q$-deformed enveloping algebra $U_q(\widehat{\mathfrak{sl}}_2)$ \cite{CP,lusztig}. The algebra $U_q^+$ is defined by the generators $A,B$ and the $q$-Serre relations 
\begin{equation}\label{eq:qSerre1}
A^3B-[3]_qA^2BA+[3]_qABA^2-BA^3=0, 
\end{equation}
\begin{equation}\label{eq:qSerre2}
B^3A-[3]_qB^2AB+[3]_qBAB^2-AB^3=0. 
\end{equation}

\noindent We will be discussing a $q$-shuffle algebra $\mV$. Let $x,y$ denote noncommuting indeterminates, and let $\mV$ denote the free algebra generated by $x,y$. We call $x$ and $y$ \textit{letters}. For $n \in \mN$, the product of $n$ letters is called a \textit{word} of \textit{length} $n$. The word of length $0$ is called \textit{trivial} and denoted by $\m1$. The words form a basis for the vector space $\mV$; this basis is called \textit{standard}. The vector space $\mV$ admits another algebra structure, called the $q$-shuffle algebra \cite{rosso1,rosso2}. The $q$-shuffle product is denoted by $\star$. The following recursive definition of $\star$ is adopted from \cite{green}. 

\begin{itemize}
\item For $v \in \mV$, 
	\begin{equation*}\label{star1}
	\m1 \star v=v \star \m1=v. 
	\end{equation*}
\item For the letters $u,v$, 
	\begin{equation*}\label{star2}
	u \star v=uv+vuq^{\langle u,v \rangle}, 
	\end{equation*}	
	where 
	 \begin{equation*}
	\langle x,x \rangle=\langle y,y \rangle =2, \hspace{4em}\langle x,y \rangle=\langle y,x \rangle=-2.
	\end{equation*}
\item For a letter $u$ and a nontrivial word $v=v_1v_2 \cdots v_n$ in $\mV$, 
	\begin{equation*}\label{star3.1}
	u \star v=\sum_{i=0}^n v_1 \cdots v_iuv_{i+1} \cdots v_n q^{\langle u,v_1 \rangle+\cdots+\langle u,v_i \rangle}, 
	\end{equation*}
	\begin{equation*}\label{star3.2}
	v \star u=\sum_{i=0}^n v_1 \cdots v_iuv_{i+1} \cdots v_n q^{\langle u,v_n \rangle+\cdots+\langle u,v_{i+1} \rangle}. 
	\end{equation*}
\item For nontrivial words $u=u_1u_2 \cdots u_r$ and $v=v_1v_2 \cdots v_s$ in $\mV$, 
	\begin{equation*}\label{star4.1}
	u \star v=u_1((u_2 \cdots u_r) \star v)+v_1(u \star (v_2 \cdots v_s))q^{\langle v_1,u_1 \rangle+\cdots+\langle v_1,u_r \rangle}, 
	\end{equation*}
	\begin{equation*}\label{star4.2}
	u \star v=(u \star (v_1 \cdots v_{s-1}))v_s+((u_1 \cdots u_{r-1}) \star v)u_rq^{\langle u_r,v_1 \rangle+\cdots+\langle u_r,v_s \rangle}. 
	\end{equation*}
\end{itemize}

\noindent It was shown in \cite{rosso1,rosso2} that the vector space $\mV$, together with the $q$-shuffle product $\star$, forms an algebra. This is the $q$-shuffle algebra. 

\medskip
\noindent Next we recall an embedding of $U_q^+$ into the $q$-shuffle algebra $\mV$. This embedding is due to Rosso \cite{rosso1,rosso2}. He showed that $x,y$ satisfy 
\begin{equation*}
x \star x \star x \star y-[3]_qx \star x \star y \star x+[3]_qx \star y \star x \star x-y \star x \star x \star x=0, 
\end{equation*}
\begin{equation*}
y \star y \star y \star x-[3]_qy \star y \star x \star y+[3]_qy \star x \star y \star y-x \star y \star y \star y=0. 
\end{equation*}
Consequentially there exists an algebra homomorphism $\natural$ from $U_q^+$ to the $q$-shuffle algebra $\mV$ that sends $A \mapsto x$ and $B \mapsto y$. It was shown in \cite[Theorem 15]{rosso2} that $\natural$ is injective. Let $U$ denote the image of $U_q^+$ under $\natural$. Observe that $U$ is the subalgebra of the $q$-shuffle algebra $\mV$ generated by $x,y$. We remark that this subalgebra is proper. Throughout this paper, we identify $U_q^+$ with $U$ via $\natural$. 

\medskip
\noindent We now discuss several types of words in $\mV$. 

\begin{definition}\label{def:balanced}\rm
(See \cite[Definition 1.3]{ter_catalan}.) A word $a_1a_2 \cdots a_n$ in $\mV$ is said to be {\it balanced} whenever
\begin{equation*}
|\{i \mid 1 \leq i \leq n, a_i = x\}|=|\{i \mid 1 \leq i \leq n, a_i = y\}|.
\end{equation*}
In this case $n$ is even. 
\end{definition}

\begin{example}\rm
We list the balanced words of length $\leq 4$. 
\begin{equation*}
\m1, \hspace{4em}xy, \hspace{1em}yx, 
\end{equation*}
\begin{equation*}
xxyy, \hspace{1em}xyxy, \hspace{1em}xyyx, \hspace{1em}yxxy, \hspace{1em}yxyx, \hspace{1em}yyxx. 
\end{equation*}
\end{example}

\begin{definition}\label{def:weight}\rm
(See \cite[Definition 1.3]{ter_catalan}.) For notational convenience, to each letter $a$ we assign a weight $\overline{a}$ as follows: 
\begin{equation*}
\overline{x}=1, \hspace{4em}\overline{y}=-1. 
\end{equation*}
\end{definition}

\begin{lemma}\label{lem:balanced}\rm
(See \cite[Definition 1.3]{ter_catalan}.) A word $a_1a_2 \cdots a_n$ is balanced if and only if $\overline{a}_1+\overline{a}_2+\cdots+\overline{a}_n=0$. 
\end{lemma}
\begin{proof}
Follows from Definitions \ref{def:balanced} and \ref{def:weight}. 
\end{proof}

\noindent We will be discussing a certain type of balanced word, said to be Catalan. 

\begin{definition}\label{def:Cat}\rm
(See \cite[Definition 1.3]{ter_catalan}.) A word $a_1a_2 \cdots a_n$ is \textit{Catalan} whenever $\overline{a}_1+\overline{a}_2+\cdots+\overline{a}_i \geq 0$ for $1 \leq i \leq n-1$ and $\overline{a}_1+\overline{a}_2+\cdots+\overline{a}_n=0$. A Catalan word is balanced. The length of a Catalan word is even. 
\end{definition}

\begin{example}\label{ex:Catn}\rm
We list the Catalan words of length $\leq 6$. 
\begin{equation*}
\m1, \hspace{4em}xy, \hspace{4em}xyxy, \hspace{1em}xxyy, 
\end{equation*}
\begin{equation*}
xyxyxy, \hspace{1em}xxyyxy, \hspace{1em}xyxxyy, \hspace{1em}xxyxyy, \hspace{1em}xxxyyy. 
\end{equation*}
\end{example}

\begin{lemma}\label{lem:Catxy}\rm
Let $w=a_1a_2 \cdots a_{2n}$ be a nontrivial Catalan word. Then $a_1=x$ and $a_{2n}=y$. 
\end{lemma}
\begin{proof}
Follows from Definition \ref{def:Cat}. 
\end{proof}

\begin{definition}\label{def:Cn}\rm
(See \cite[Definition 1.5]{ter_catalan}.) For $n \in \mN$, define 
\begin{equation*}
C_n=\sum a_1a_2 \cdots a_{2n}\prod_{i=1}^{2n}[1+\overline{a}_1+\overline{a}_2+\cdots+\overline{a}_i]_q, 
\end{equation*}
where the sum is over all Catalan words $a_1a_2 \cdots a_{2n}$ of length $2n$. We interpret $C_0=\m1$. We call $C_n$ the $n^{\rm th}$ Catalan element. 
\end{definition}

\begin{example}\label{ex:Cn}\rm
We list $C_n$ for $0 \leq n \leq 3$. 
\begin{equation*}
C_0=\m1, \hspace{4em} C_1=[2]_qxy, \hspace{4em} C_2=[2]_q^2xyxy+[2]_q^2[3]_qxxyy, 
\end{equation*}
\begin{equation*}
C_3=[2]_q^3xyxyxy+[2]_q^3[3]_qxxyyxy+[2]_q^3[3]_qxyxxyy+[2]_q^3[3]_q^2xxyxyy+[2]_q^2[3]_q^2[4]_qxxxyyy. 
\end{equation*}
\end{example}

\noindent We will be discussing the notion of a PBW basis. We refer to \cite[Definition 2.1]{ter_alternating} for the definition of this notion. 

\medskip
\noindent Next we recall three PBW basis for $U$. We will encounter the elements $xC_n$, $C_ny$, $xC_{n-1}y$. We emphasize that this notation refers to the free product on $\mV$. 

\begin{proposition}\label{prop:CnPBW}\rm
(See \cite[Corollary 6.8]{ter_beck}.) The elements $\{xC_n\}_{n \in \mN}$, $\{C_ny\}_{n \in \mN}$, $\{C_n\}_{n=1}^\infty$ form a PBW basis for $U$ under the linear ordering 
\begin{equation*}
xC_0<xC_1<xC_2<\cdots<C_1<C_2<C_3<\cdots<C_2y<C_1y<C_0y. 
\end{equation*}
\end{proposition}

\noindent We remark that up to normalization, the above PBW basis is the one given by Damiani in \cite{damiani}. 

\begin{proposition}\label{prop:xCnyPBW}\rm
(See \cite[Corollary 8.3]{ter_beck}.) The elements $\{xC_n\}_{n \in \mN}$, $\{C_ny\}_{n \in \mN}$, $\{xC_ny\}_{n \in \mN}$ form a PBW basis for $U$ under the linear ordering 
\begin{equation*}
xC_0<xC_1<xC_2<\cdots<xC_0y<xC_1y<xC_2y<\cdots<C_2y<C_1y<C_0y. 
\end{equation*}
\end{proposition}

\noindent We remark that up to normalization, the above PBW basis is the one given by Beck in \cite{beck}. 

\begin{definition}\label{def:tGn}\rm
(See \cite[Definition 5.2]{ter_alternating}.) For $n \in \mN$, define
\begin{equation*}
W_{-n}=(xy)^nx, \hspace{2em}W_{n+1}=y(xy)^n, \hspace{2em}G_{n+1}=(yx)^{n+1}, \hspace{2em}\tG_{n+1}=(xy)^{n+1}. 
\end{equation*}
The above exponents are with respect to the free product. 
\end{definition}

\begin{example}\label{ex:tGn}\rm
We list $W_{-n}$, $W_{n+1}$, $G_{n+1}$, $\tG_{n+1}$ for $0 \leq n \leq 3$. 
\begin{align*}
&W_0=x, \hspace{2em} W_{-1}=xyx, \hspace{2em} W_{-2}=xyxyx, \hspace{2em} W_{-3}=xyxyxyx; \\
&W_1=y, \hspace{2em} W_2=yxy, \hspace{2em} W_3=yxyxy, \hspace{2em} W_4=yxyxyxy; \\
&G_1=yx, \hspace{2em} G_2=yxyx, \hspace{2em} G_3=yxyxyx, \hspace{2em} G_4=yxyxyxyx; \\
&\tG_1=xy, \hspace{2em} \tG_2=xyxy, \hspace{2em} \tG_3=xyxyxy, \hspace{2em} \tG_4=xyxyxyxy. 
\end{align*}
\end{example}

\noindent The words $\{W_{-n}\}_{n \in \mN}$, $\{W_{n+1}\}_{n \in \mN}$, $\{G_{n+1}\}_{n \in \mN}$, $\{\tG_{n+1}\}_{n \in \mN}$ are called \textit{alternating}. For notational convenience, we define $G_0=\m1$ and $\tG_0=\m1$. 

\begin{proposition}\label{prop:tGnPBW}\rm
(See \cite[Theorems 10.1 and 10.2]{ter_alternating}.) Each of the following (i), (ii) forms a PBW basis for $U$ under an appropriate linear ordering. 
\begin{enumerate}
\item $\{W_{-n}\}_{n \in \mN}$, $\{W_{n+1}\}_{n \in \mN}$, $\{\tG_{n+1}\}_{n \in \mN}$; 
\item $\{W_{-n}\}_{n \in \mN}$, $\{W_{n+1}\}_{n \in \mN}$, $\{G_{n+1}\}_{n \in \mN}$. 
\end{enumerate}
\end{proposition}

\noindent In \cite[Section 9]{ter_alternating} it is explained how the two PBW bases in Proposition \ref{prop:tGnPBW} are related. This relationship is described using some generating functions. We have some comments on these generating functions. 

\medskip
\noindent For the rest of this paper, let $t$ denote an indeterminate. We will discuss generating functions in the variable $t$.

\begin{definition}\label{intro:tG(t)}\rm
Define the generating function 
\begin{equation*}
\tG(t)=\sum_{n \in \mN}\tG_nt^n. 
\end{equation*}
\end{definition}

\noindent by \cite[Lemma 1.4]{inverse}, the function $\tG(t)$ is invertible with respect to the $q$-shuffle product. 

\begin{definition}\label{intro:D(t)}\rm
Let $D(t)$ denote the multiplicative inverse of $\tG(t)$ with respect to the $q$-shuffle product. In other words, 
\begin{equation}\label{eq:tG*D}
\tG(t) \star D(t)=\m1=D(t) \star \tG(t). 
\end{equation}
Write 
\begin{equation*}
D(t)=\sum_{n \in \mN}D_nt^n. 
\end{equation*}
Note that $D_0=\m1$. 
\end{definition}

\noindent The elements $\{D_n\}_{n \in \mN}$ can be computed recursively using \eqref{eq:tG*D}. A closed form for $\{D_n\}_{n \in \mN}$ was given in \cite{inverse}. In a moment we will describe this closed form. Observe that 
\begin{equation}\label{eq:xy01}
(\overline{x}+1)/2=1, \hspace{4em} (\overline{y}+1)/2=0. 
\end{equation}

\begin{proposition}\label{prop:Dn}\rm
(See \cite[Definition 1.9 and Theorem 1.11]{inverse}.) For $n \in \mN$, 
\begin{equation*}
D_n=(-1)^n\sum a_1a_2 \cdots a_{2n}\prod_{i=1}^{2n}[\overline{a}_1+\overline{a}_2+\cdots+\overline{a}_{i-1}+(\overline{a}_i+1)/2]_q, 
\end{equation*}
where the sum is over all Catalan words $a_1a_2 \cdots a_{2n}$ of length $2n$. 
\end{proposition}

\begin{example}\label{ex:Dn}\rm
We list $D_n$ for $0 \leq n \leq 3$. 
\begin{equation*}
D_0=\m1, \hspace{4em} D_1=-xy, \hspace{4em} D_2=xyxy+[2]_q^2xxyy, 
\end{equation*}
\begin{equation*}
D_3=-xyxyxy-[2]_q^2xxyyxy-[2]_q^2xyxxyy-[2]_q^4xxyxyy-[2]_q^2[3]_q^2xxxyyy. 
\end{equation*}
\end{example}

\noindent To motivate our main results, we recall some relations involving the elements $\{C_n\}_{n \in \mN}$, $\{D_n\}_{n \in \mN}$, $\{\tG_n\}_{n \in \mN}$, $\{xC_ny\}_{n \in \mN}$. In order to do this, we bring in another generating function. 

\begin{definition}\label{intro:C(t)}\rm
Define the generating function 
\begin{equation*}
C(t)=\sum_{n \in \mN}C_nt^n. 
\end{equation*}
\end{definition}

\noindent We will be discussing the exponential function 
\begin{equation*}
\exp(z)=\sum_{n \in \mN}\frac{z^n}{n!}. 
\end{equation*}

\begin{theorem}\label{thm:genfuns}\rm
(See \cite[Corollaries 8.1, 8.4 and Proposition 9.11]{ter_beck}.) The following (i)--(iv) hold. 
\begin{enumerate}
\item The elements $\{xC_ny\}_{n \in \mN}$ mutually commute with respect to the $q$-shuffle product. 
\item $\displaystyle C(t)=\exp\left(\sum_{n=1}^\infty \frac{[2n]_q}{n}xC_{n-1}yt^n\right)$. 
\item $\displaystyle \tG(t)=\exp\left(-\sum_{n=1}^\infty \frac{(-1)^n[n]_q}{n}xC_{n-1}yt^n\right)$. 
\item $\displaystyle D(t)=\exp\left(\sum_{n=1}^\infty \frac{(-1)^n[n]_q}{n}xC_{n-1}yt^n\right)$. 
\end{enumerate}
\end{theorem}

\noindent In the above equations, the exponential power series is computed with respect to the $q$-shuffle product. 

\medskip
\noindent In the next result we give a variation on the exponential formulas from Theorem \ref{thm:genfuns}. 

\begin{proposition}\label{prop:expderivative}\rm
For $n \geq 1$, 
\begin{enumerate}
\item $\displaystyle C_n=\frac{1}{n}\sum_{k=1}^n [2k]_qxC_{k-1}y \star C_{n-k}$; 
\item $\displaystyle \tG_n=-\frac{1}{n}\sum_{k=1}^n (-1)^k[k]_qxC_{k-1}y \star \tG_{n-k}$; 
\item $\displaystyle D_n=\frac{1}{n}\sum_{k=1}^n (-1)^k[k]_qxC_{k-1}y \star D_{n-k}$. 
\end{enumerate}
\end{proposition}
\begin{proof}
(i) Taking the derivative with respect to $t$ on both sides of Theorem \ref{thm:genfuns}(ii) gives 
\begin{equation*}
\sum_{n=1}^\infty nC_nt^{n-1}=\left(\sum_{n=1}^\infty [2n]_qxC_{n-1}yt^{n-1}\right) \star C(t). 
\end{equation*}
Comparing the coefficients of $t^{n-1}$ in the above equation gives the desired result. 

\medskip
\noindent (ii), (iii) Similar to the proof of (i). 
\end{proof}

\noindent The formulas in Proposition \ref{prop:expderivative} can be used to recursively compute $\{C_n\}_{n=1}^\infty$, $\{\tG_n\}_{n=1}^\infty$, $\{D_n\}_{n=1}^\infty$ in terms of $\{xC_ny\}_{n \in \mN}$. 

\medskip
\noindent The following propositions are immediate consequences of Theorem \ref{thm:genfuns}. 

\begin{proposition}\label{prop:genfuns2}\rm
(See \cite[Corollary 1.8]{ter_catalan}, \cite[Proposition 5.10 and Lemma 9.10]{ter_alternating}, \cite[Corollary 8.4]{ter_beck}.) The elements in the set 
\begin{equation*}
\{C_n\}_{n \in \mN} \cup \{\tG_n\}_{n \in \mN} \cup \{D_n\}_{n \in \mN} \cup \{xC_ny\}_{n \in \mN} 
\end{equation*}
mutually commute with respect to the $q$-shuffle product. 
\end{proposition}
\begin{proof}
Follows from Theorem \ref{thm:genfuns}. 
\end{proof}

\begin{proposition}\label{prop:genfuns3}\rm
(See \cite[Proposition 11.8]{ter_alternating}.) We have 
\begin{equation*}
C(-t)=D(qt) \star D(q^{-1}t). 
\end{equation*}
\end{proposition}
\begin{proof}
Follows from Theorem \ref{thm:genfuns}. 
\end{proof}

\noindent In this paper we have two main results. We now state our first main result. 

\begin{theorem}\label{thm:genfuns4}\rm
For $m \geq 1$, 
\begin{equation}\label{eq:tG=exp}
\tG(-q^{m-1}t) \star \tG(-q^{m-3}t) \star \cdots \star \tG(-q^{1-m}t)=\exp\left(-\sum_{n=1}^\infty \frac{[mn]_q}{n}xC_{n-1}yt^n\right); 
\end{equation}
\begin{equation}\label{eq:D=exp}
D(-q^{m-1}t) \star D(-q^{m-3}t) \star \cdots \star D(-q^{1-m}t)=\exp\left(\sum_{n=1}^\infty \frac{[mn]_q}{n}xC_{n-1}yt^n\right). 
\end{equation}
In the above equations, the exponential power series is computed with respect to the $q$-shuffle product. 
\end{theorem}

\noindent We now state our second main result. 

\begin{theorem}\label{thm:closedform}\rm
For $m \geq 1$ the following (i), (ii) hold. 
\begin{enumerate}
\item For $n \in \mN$, the coefficient of $t^n$ in either side of \eqref{eq:tG=exp} is 
\begin{equation}\label{eq:coeff-m}
\sum a_1a_2 \cdots a_{2n}\prod_{i=1}^{2n}[\overline{a}_1+\overline{a}_2+\cdots+\overline{a}_{i-1}-m(\overline{a}_i+1)/2]_q, 
\end{equation}
where the sum is over all Catalan words $a_1a_2 \cdots a_{2n}$ of length $2n$. 
\item For $n \in \mN$, the coefficient of $t^n$ in either side of \eqref{eq:D=exp} is 
\begin{equation}\label{eq:coeff+m}
\sum a_1a_2 \cdots a_{2n}\prod_{i=1}^{2n}[\overline{a}_1+\overline{a}_2+\cdots+\overline{a}_{i-1}+m(\overline{a}_i+1)/2]_q, 
\end{equation}
where the sum is over all Catalan words $a_1a_2 \cdots a_{2n}$ of length $2n$. 
\end{enumerate}
\end{theorem}

\noindent In this paper, we will prove Theorems \ref{thm:genfuns4} and \ref{thm:closedform}. Our proof is from scratch. We will not cite earlier results from the literature. The proof for Theorems \ref{thm:genfuns4} and \ref{thm:closedform} takes up most of this paper. In the course of this proof, we will prove Theorem \ref{thm:genfuns} from scratch, without citing earlier results from the literature. 

\medskip
\noindent In order to prove Theorems \ref{thm:genfuns}, \ref{thm:genfuns4}, \ref{thm:closedform}, we will use the following strategy. We will introduce two types of elements in $\mV$, denoted by $\{\Delta^{(m)}_n\}_{m \in \mZ, n \in \mN}$ and $\{\nabla^{(m)}_n\}_{m \in \mZ, n \geq 1}$. We will develop a uniform theory of these elements. This theory will be used to establish Theorems \ref{thm:genfuns}, \ref{thm:genfuns4}, \ref{thm:closedform} from scratch. The elements $\{\Delta^{(m)}_n\}_{m \in \mZ, n \in \mN}$ and $\{\nabla^{(m)}_n\}_{m \in \mZ, n \geq 1}$ will be introduced in Sections 4 and 5. 

\section{Elevation sequences, Dyck paths, and profiles}
In this section, we recall a few notions that will be useful. 

\begin{definition}\label{def:elevation_sequence}\rm
(See \cite[Definition 2.6]{ter_catalan}.) For $n \in \mN$ and a word $w=a_1a_2 \cdots a_{n}$, the \textit{elevation sequence} of $w$ is the sequence $(e_0,e_1,\ldots,e_{n})$, where $e_i=\overline{a}_1+\overline{a}_2+\cdots+\overline{a}_i$ for $0 \leq i \leq n$. 
\end{definition}

\begin{example}\rm
In the table below, we list the Catalan words $w$ of length $\leq 6$ and the corresponding elevation sequences. 
\begin{center}
\begin{tabular}{ c|c } 
$w$ & elevation sequence of $w$\\
\hline
$\m1$ & $(0)$\\
$xy$ & $(0,1,0)$\\
$xyxy$ & $(0,1,0,1,0)$\\
$xxyy$ & $(0,1,2,1,0)$\\
$xyxyxy$ & $(0,1,0,1,0,1,0)$\\
$xxyyxy$ & $(0,1,2,1,0,1,0)$\\
$xyxxyy$ & $(0,1,0,1,2,1,0)$\\
$xxyxyy$ & $(0,1,2,1,2,1,0)$\\
$xxxyyy$ & $(0,1,2,3,2,1,0)$
\end{tabular}
\end{center}
\end{example}

\noindent Referring to Definition \ref{def:elevation_sequence}, we have $e_0=0$. 
\begin{lemma}\label{lem:eleCat}\rm
Let $w$ be a word with elevation sequence $(e_0,e_1,\ldots,e_n)$. The word $w$ is Catalan if and only if $e_i \geq 0$ for $1 \leq i \leq n-1$ and $e_n=0$. 
\end{lemma}
\begin{proof}
Follows from Definitions \ref{def:Cat} and \ref{def:elevation_sequence}. 
\end{proof}

\noindent Next we give a way to visualize a word using its elevation sequence. 

\begin{definition}\label{def:Dyck}\rm
(See \cite[Section 8.5]{brualdi}.) Let $n \in \mN$. For a word with elevation sequence $(e_0,e_1,\ldots,e_n)$, the corresponding \textit{Dyck path} is a diagonal lattice path with $n+1$ vertices, where for $0 \leq i \leq n$ the $i$-th vertex is the lattice point $(i,e_i)$. 
\end{definition}

\begin{example}\label{ex:DyckCat}\rm
In the picture below, we give the Dyck paths for the Catalan words of length $\leq 6$. 
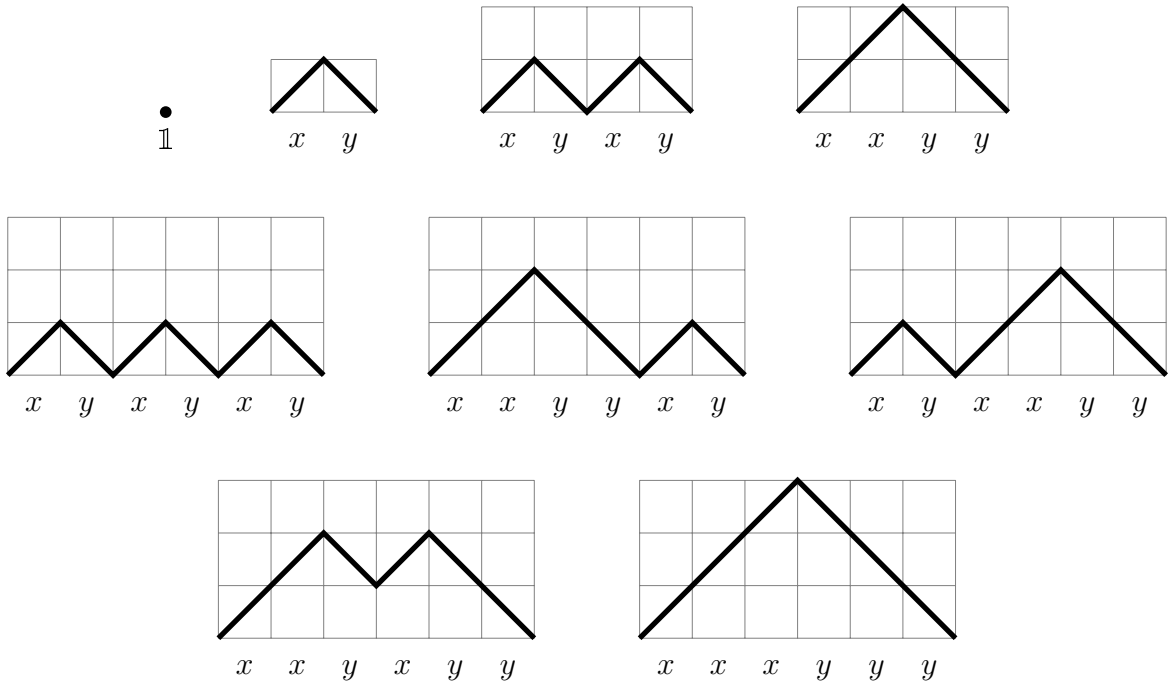
\begin{center}
\begin{tikzpicture}[scale=0.7]
\scoped[xshift=3cm]
{
\node[scale=0.4, circle, fill, label={270:$\m1$}] at (0,0) {};
}

\scoped[xshift=5cm]
{
\draw[help lines] (0,0) grid +(2,1);
\draw[line width=2pt] (0,0)--(1,1)--(2,0);
\node[label={270:$x$}] at (0.5,0) {};
\node[label={270:$y$}] at (1.5,0) {};
}

\scoped[xshift=9cm]
{
\draw[help lines] (0,0) grid +(4,2);
\draw[line width=2pt] (0,0)--(1,1)--(2,0)--(3,1)--(4,0);
\node[label={270:$x$}] at (0.5,0) {};
\node[label={270:$y$}] at (1.5,0) {};
\node[label={270:$x$}] at (2.5,0) {};
\node[label={270:$y$}] at (3.5,0) {};
}

\scoped[xshift=15cm]
{
\draw[help lines] (0,0) grid +(4,2);
\draw[line width=2pt] (0,0)--(1,1)--(2,2)--(3,1)--(4,0);
\node[label={270:$x$}] at (0.5,0) {};
\node[label={270:$x$}] at (1.5,0) {};
\node[label={270:$y$}] at (2.5,0) {};
\node[label={270:$y$}] at (3.5,0) {};
}

\scoped[yshift=-5cm]
{
\draw[help lines] (0,0) grid +(6,3);
\draw[color=black, line width=2] (0,0)--(1,1)--(2,0)--(3,1)--(4,0)--(5,1)--(6,0);
\node[label={270:$x$}] at (0.5,0) {};
\node[label={270:$y$}] at (1.5,0) {};
\node[label={270:$x$}] at (2.5,0) {};
\node[label={270:$y$}] at (3.5,0) {};
\node[label={270:$x$}] at (4.5,0) {};
\node[label={270:$y$}] at (5.5,0) {};
}

\scoped[xshift=8cm, yshift=-5cm]
{
\draw[help lines] (0,0) grid +(6,3);
\draw[color=black, line width=2] (0,0)--(1,1)--(2,2)--(3,1)--(4,0)--(5,1)--(6,0);
\node[label={270:$x$}] at (0.5,0) {};
\node[label={270:$x$}] at (1.5,0) {};
\node[label={270:$y$}] at (2.5,0) {};
\node[label={270:$y$}] at (3.5,0) {};
\node[label={270:$x$}] at (4.5,0) {};
\node[label={270:$y$}] at (5.5,0) {};
}

\scoped[xshift=16cm, yshift=-5cm]
{
\draw[help lines] (0,0) grid +(6,3);
\draw[color=black, line width=2] (0,0)--(1,1)--(2,0)--(3,1)--(4,2)--(5,1)--(6,0);
\node[label={270:$x$}] at (0.5,0) {};
\node[label={270:$y$}] at (1.5,0) {};
\node[label={270:$x$}] at (2.5,0) {};
\node[label={270:$x$}] at (3.5,0) {};
\node[label={270:$y$}] at (4.5,0) {};
\node[label={270:$y$}] at (5.5,0) {};
}

\scoped[xshift=4cm, yshift=-10cm]
{
\draw[help lines] (0,0) grid +(6,3);
\draw[color=black, line width=2] (0,0)--(1,1)--(2,2)--(3,1)--(4,2)--(5,1)--(6,0);
\node[label={270:$x$}] at (0.5,0) {};
\node[label={270:$x$}] at (1.5,0) {};
\node[label={270:$y$}] at (2.5,0) {};
\node[label={270:$x$}] at (3.5,0) {};
\node[label={270:$y$}] at (4.5,0) {};
\node[label={270:$y$}] at (5.5,0) {};
}

\scoped[xshift=12cm, yshift=-10cm]
{
\draw[help lines] (0,0) grid +(6,3);
\draw[color=black, line width=2] (0,0)--(1,1)--(2,2)--(3,3)--(4,2)--(5,1)--(6,0);
\node[label={270:$x$}] at (0.5,0) {};
\node[label={270:$x$}] at (1.5,0) {};
\node[label={270:$x$}] at (2.5,0) {};
\node[label={270:$y$}] at (3.5,0) {};
\node[label={270:$y$}] at (4.5,0) {};
\node[label={270:$y$}] at (5.5,0) {};
}
\end{tikzpicture}
\end{center}
\end{example}

\noindent Observe in Example \ref{ex:DyckCat} that the Dyck path of a word is uniquely determined by its peaks and valleys. This observation is captured in the following definition. 

\begin{definition}\label{def:profile}\rm
(See \cite[Definition 2.8]{ter_catalan}.) For a word $w$ with elevation sequence $(e_0,e_1,\ldots,e_n)$, the \textit{profile} of $w$ is the subsequence of $(e_0,e_1,\ldots,e_n)$ obtained by removing all the $e_i$ such that $1 \leq i \leq n-1$ and $e_i-e_{i-1}=e_{i+1}-e_i$. 

\medskip
\noindent In other words, the profile of a word $w$ is the subsequence of the elevation sequence of $w$ consisting of the end points and turning points. 

\medskip
\noindent By a \textit{profile} we mean the profile of a word. 

\medskip
\noindent By a \textit{nontrivial profile} (resp. \textit{Catalan profile}) we mean the profile of a nontrivial word (resp. Catalan word). 
\end{definition}

\begin{example}\rm
In the table below, we list the Catalan words $w$ of length $\leq 6$ and the corresponding profiles. 
\begin{center}
\begin{tabular}{ c|c }
$w$ & profile of $w$\\
\hline
$\m1$ & $(0)$\\
$xy$ & $(0,1,0)$\\
$xyxy$ & $(0,1,0,1,0)$\\
$xxyy$ & $(0,2,0)$\\
$xyxyxy$ & $(0,1,0,1,0,1,0)$\\
$xxyyxy$ & $(0,2,0,1,0)$\\
$xyxxyy$ & $(0,1,0,2,0)$\\
$xxyxyy$ & $(0,2,1,2,0)$\\
$xxxyyy$ & $(0,3,0)$\\
\end{tabular}
\end{center}
\end{example}

\begin{lemma}\label{lem:profileCat}\rm
A profile $(l_0,h_1,l_1,\ldots,h_r,l_r)$ is Catalan if and only if $l_i \geq 0$ for $1 \leq i \leq r-1$ and $l_r=0$. 
\end{lemma}
\begin{proof}
Follows from Lemma \ref{lem:eleCat} and Definition \ref{def:profile}. 
\end{proof}

\noindent We end this section with an observation. 

\begin{lemma}\label{lem:w=tG}\rm
Let $n \in \mN$. For a word $w$ of length $2n$, the following are equivalent: 
\begin{enumerate}
\item $w=\tG_n$; 
\item each entry in the elevation sequence of $w$ is either $0$ or $1$. 
\end{enumerate}
\end{lemma}
\begin{proof}
Follows from Definitions \ref{def:tGn} and \ref{def:elevation_sequence}. 
\end{proof}

\section{The elements $\Delta^{(m)}_n$}
In this section, we introduce the elements $\Delta^{(m)}_n$. These elements are inspired by \eqref{eq:coeff-m} and \eqref{eq:coeff+m}. 

\medskip
\noindent In \eqref{eq:coeff-m} and \eqref{eq:coeff+m} it was assumed that $m \geq 1$. In the following definition, it is convenient to assume $m \in \mZ$. 

\begin{definition}\label{def:D^m(w)}\rm
Let $m \in \mZ$. For $n \in \mN$ and a Catalan word $w=a_1a_2 \cdots a_{2n}$, define the scalar 
\begin{equation*}
\Delta^{(m)}(w)=\prod_{i=1}^{2n}[\overline{a}_1+\overline{a}_2+\cdots+\overline{a}_{i-1}+m(\overline{a}_i+1)/2]_q. 
\end{equation*}
For $n=0$, we have $w=\m1$. In this case, $\Delta^{(m)}(w)=1$. 
\end{definition}

\noindent The following formula will be useful. 

\begin{lemma}\label{lem:D^m(w)}\rm
Let $m \in \mZ$. For $n \in \mN$ and a Catalan word $w=a_1a_2 \cdots a_{2n}$, 
\begin{equation}\label{eq:D^m(w)xy}
\Delta^{(m)}(w)=\left(\prod_{\substack{1 \leq i \leq 2n\\a_i=x}}[\overline{a}_1+\overline{a}_2+\cdots+\overline{a}_{i-1}+m]_q\right)\left(\prod_{\substack{1 \leq i \leq 2n\\a_i=y}}[\overline{a}_1+\overline{a}_2+\cdots+\overline{a}_{i-1}]_q\right). 
\end{equation}
\end{lemma}
\begin{proof}
Follows from Definition \ref{def:D^m(w)} and \eqref{eq:xy01}. 
\end{proof}

\begin{example}\rm
In the table below, we list the Catalan words $w$ of length $\leq 6$ and the corresponding scalars $\Delta^{(m)}(w)$ for $-3 \leq m \leq 3$. 
\begin{center}
\begin{tabular}{ c|c|c|c|c|c|c|c }
$w$ & $\Delta^{(-3)}(w)$ & $\Delta^{(-2)}(w)$ & $\Delta^{(-1)}(w)$ & $\Delta^{(0)}(w)$ & $\Delta^{(1)}(w)$ & $\Delta^{(2)}(w)$ & $\Delta^{(3)}(w)$\\[1mm]
\hline
$\m1$ & $1$ & $1$ & $1$ & $1$ & $1$ & $1$ & $1$\\[1mm]
$xy$ & $-[3]_q$ & $-[2]_q$ & $-1$ & $0$ & $1$ & $[2]_q$ & $[3]_q$\\[1mm]
$xyxy$ & $[3]_q^2$ & $[2]_q^2$ & $1$ & $0$ & $1$ & $[2]_q^2$ & $[3]_q^2$\\[1mm]
$xxyy$ & $[2]_q^2[3]_q$ & $[2]_q^2$ & $0$ & $0$ & $[2]_q^2$ & $[2]_q^2[3]_q$ & $[2]_q[3]_q[4]_q$\\[1mm]
$xyxyxy$ & $-[3]_q^3$ & $-[2]_q^3$ & $-1$ & $0$ & $1$ & $[2]_q^3$ & $[3]_q^3$\\[1mm]
$xxyyxy$ & $-[2]_q^2[3]_q^2$ & $-[2]_q^3$ & $0$ & $0$ & $[2]_q^2$ & $[2]_q^3[3]_q$ & $[2]_q[3]_q^2[4]_q$\\[1mm]
$xyxxyy$ & $-[2]_q^2[3]_q^2$ & $-[2]_q^3$ & $0$ & $0$ & $[2]_q^2$ & $[2]_q^3[3]_q$ & $[2]_q[3]_q^2[4]_q$\\[1mm]
$xxyxyy$ & $-[2]_q^4[3]_q$ & $-[2]_q^3$ & $0$ & $0$ & $[2]_q^4$ & $[2]_q^3[3]_q^2$ & $[2]_q^2[3]_q[4]_q^2$\\[1mm]
$xxxyyy$ & $-[2]_q^2[3]_q^2$ & $0$ & $0$ & $0$ & $[2]_q^2[3]_q^2$ & $[2]_q^2[3]_q^2[4]_q$ & $[2]_q[3]_q^2[4]_q[5]_q$
\end{tabular}
\end{center}
\end{example}

\begin{lemma}\label{lem:D^m(w)=0}\rm
Let $m \leq -1$. For $n \in \mN$ and a Catalan word $w=a_1a_2 \cdots a_{2n}$, we have $\Delta^{(m)}(w) \neq 0$ if and only if $\overline{a}_1+\overline{a}_2+\cdots+\overline{a}_i \leq |m|$ for $0 \leq i \leq 2n$. 
\end{lemma}
\begin{proof}
Note that there exists an integer $i$ ($0 \leq i \leq 2n$) such that $\overline{a}_1+\overline{a}_2+\cdots+\overline{a}_i \geq |m|+1$ if and only if there exists an integer $j$ ($1 \leq j \leq 2n$) such that $a_j=x$ and $\overline{a}_1+\overline{a}_2+\cdots+\overline{a}_{j-1}=|m|$. Evaluating \eqref{eq:D^m(w)xy} using this comment, we obtain the result. 
\end{proof}

\noindent The following notation will be useful. 

\begin{definition}\label{def:Cat_n}\rm
For $n \in \mN$, let $\Cat_n$ denote the set of Catalan words of length $2n$. 
\end{definition}

\begin{definition}\label{def:D^m_n}\rm
For $m \in \mZ$ and $n \in \mN$, define 
\begin{equation*}
\Delta^{(m)}_n=\sum_{w \in \Cat_n}\Delta^{(m)}(w)w. 
\end{equation*}
We remark that $\Delta^{(m)}_0=\m1$. 
\end{definition}

\noindent Next, we examine the cases $-1 \leq m \leq 2$ in Definition \ref{def:D^m_n}. For convenience, we do this in order $m=2,1,-1,0$. 

\begin{lemma}\label{lem:D^2_n}\rm
For $n \in \mN$, 
\begin{equation*}
\Delta^{(2)}_n=C_n. 
\end{equation*}
\end{lemma}
\begin{proof}
Follows from Definitions \ref{def:Cn} and \ref{def:D^m_n}. 
\end{proof}

\begin{example}\rm
We will show later in this paper that for $n \in \mN$, 
\begin{equation*}
\Delta^{(1)}_n=(-1)^nD_n. 
\end{equation*}
\end{example}

\begin{lemma}\label{lem:D^-1_n}\rm
For $n \in \mN$, 
\begin{equation*}
\Delta^{(-1)}_n=(-1)^n\tG_n. 
\end{equation*}
\end{lemma}
\begin{proof}
For $w \in \Cat_n$, by Lemmas \ref{lem:w=tG} and \ref{lem:D^m(w)=0} we have $\Delta^{(-1)}(w) \neq 0$ if and only if the elevation sequence of $w$ consists of $0$ and $1$ if and only if $w=\tG_n$. In this case, $\Delta^{(-1)}(w)=(-1)^n$, and the result follows by Definition \ref{def:D^m_n}. 
\end{proof}

\begin{lemma}\label{lem:D^0_n}\rm
For $n \in \mN$, 
\begin{equation*}
\Delta^{(0)}_n=\begin{cases}\m1, &\hspace{1em}\text{ if }n=0;\\0, &\hspace{1em}\text{ if }n \geq 1.\end{cases}
\end{equation*}
\end{lemma}
\begin{proof}
The case $n=0$ is trivial. 

\medskip 
\noindent Now assume $n \geq 1$. For a Catalan word $w=a_1a_2 \cdots a_{2n}$, by Definition \ref{def:D^m(w)} we have 
\begin{equation*}
\Delta^{(0)}(w)=\prod_{i=1}^{2n}[\overline{a}_1+\overline{a}_2+\cdots+\overline{a}_{i-1}]_q=0, 
\end{equation*}
where the second equality holds because the first factor in the product is $[0]_q$, and $[0]_q=0$. The result follows by Definition \ref{def:D^m_n}. 
\end{proof}

\section{The elements $\nabla^{(m)}_n$}
In this section, we introduce a variation on $\Delta^{(m)}_n$ called $\nabla^{(m)}_n$. 

\begin{definition}\label{def:N^m(w)}\rm
Let $m \in \mZ$. For $n \geq 1$ and a Catalan word $w=a_1a_2 \cdots a_{2n}$, define the scalar 
\begin{equation*}
\nabla^{(m)}(w)=\prod_{i=2}^{2n}[\overline{a}_1+\overline{a}_2+\cdots+\overline{a}_{i-1}+m(\overline{a}_i+1)/2]_q. 
\end{equation*}
\end{definition}

\noindent We emphasize that $\nabla^{(m)}(w)$ is not defined for $w=\m1$. 

\begin{example}\rm
In the table below, we list the nontrivial Catalan words $w$ of length $\leq 6$ and the corresponding scalars $\nabla^{(m)}(w)$ for $-3 \leq m \leq 3$. 
\begin{center}
\begin{tabular}{ c|c|c|c|c|c|c|c }
$w$ & $\nabla^{(-3)}(w)$ & $\nabla^{(-2)}(w)$ & $\nabla^{(-1)}(w)$ & $\nabla^{(0)}(w)$ & $\nabla^{(1)}(w)$ & $\nabla^{(2)}(w)$ & $\nabla^{(3)}(w)$\\[1mm]
\hline
$xy$ & $1$ & $1$ & $1$ & $1$ & $1$ & $1$ & $1$\\[1mm]
$xyxy$ & $-[3]_q$ & $-[2]_q$ & $-1$ & $0$ & $1$ & $[2]_q$ & $[3]_q$\\[1mm]
$xxyy$ & $-[2]_q^2$ & $-[2]_q$ & $0$ & $[2]_q$ & $[2]_q^2$ & $[2]_q[3]_q$ & $[2]_q[4]_q$\\[1mm]
$xyxyxy$ & $[3]_q^2$ & $[2]_q^2$ & $1$ & $0$ & $1$ & $[2]_q^2$ & $[3]_q^2$\\[1mm]
$xxyyxy$ & $[2]_q^2[3]_q$ & $[2]_q^2$ & $0$ & $0$ & $[2]_q^2$ & $[2]_q^2[3]_q$ & $[2]_q[3]_q[4]_q$\\[1mm]
$xyxxyy$ & $[2]_q^2[3]_q$ & $[2]_q^2$ & $0$ & $0$ & $[2]_q^2$ & $[2]_q^2[3]_q$ & $[2]_q[3]_q[4]_q$\\[1mm]
$xxyxyy$ & $[2]_q^4$ & $[2]_q^2$ & $0$ & $[2]_q^2$ & $[2]_q^4$ & $[2]_q^2[3]_q^2$ & $[2]_q^2[4]_q^2$\\[1mm]
$xxxyyy$ & $[2]_q^2[3]_q$ & $0$ & $0$ & $[2]_q^2[3]_q$ & $[2]_q^2[3]_q^2$ & $[2]_q[3]_q^2[4]_q$ & $[2]_q[3]_q[4]_q[5]_q$
\end{tabular}
\end{center}
\end{example}

\medskip
\noindent In the following lemma, we compare $\nabla^{(m)}(w)$ and $\Delta^{(m)}(w)$ for a nontrivial Catalan word $w$. 

\begin{lemma}\label{lem:D&N(w)}\rm
For $m \in \mZ$ and a nontrivial Catalan word $w$, 
\begin{equation*}
\Delta^{(m)}(w)=[m]_q\nabla^{(m)}(w). 
\end{equation*}
\end{lemma}
\begin{proof}
Follows from Definitions \ref{def:D^m(w)} and \ref{def:N^m(w)}. 
\end{proof}

\begin{definition}\label{def:N^m_n}\rm
For $m \in \mZ$ and $n \geq 1$, define 
\begin{equation*}
\nabla^{(m)}_n=\sum_{w \in \Cat_n}\nabla^{(m)}(w)w. 
\end{equation*}
\end{definition}

\noindent We have a comment on Definition \ref{def:N^m_n}. 

\begin{lemma}\label{lem:N^m_1}\rm
For $m \in \mZ$, 
\begin{equation*}
\nabla^{(m)}_1=xy. 
\end{equation*}
\end{lemma}
\begin{proof}
Follows from Definitions \ref{def:N^m(w)} and \ref{def:N^m_n}. 
\end{proof}

\noindent In the following lemma we compare $\nabla^{(m)}_n$ and $\Delta^{(m)}_n$. 

\begin{lemma}\label{lem:D&N}\rm
For $m \in \mZ$ and $n \geq 1$, 
\begin{equation*}
\Delta^{(m)}_n=[m]_q\nabla^{(m)}_n. 
\end{equation*}
\end{lemma}
\begin{proof}
Follows from Definitions \ref{def:D^m_n}, \ref{def:N^m_n} and Lemma \ref{lem:D&N(w)}. 
\end{proof}

\noindent We have seen that for $m \neq 0$, the elements $\Delta^{(m)}_n$ and $\nabla^{(m)}_n$ agree up to a nonzero scalar factor. Let us examine the case $m=0$. The elements $\Delta^{(0)}_n$ are given in Lemma \ref{lem:D^0_n}. Next we display $\nabla^{(0)}_n$. 

\begin{lemma}\label{lem:N^0_n}\rm
For $n \geq 1$, 
\begin{equation*}
\nabla^{(0)}_n=xC_{n-1}y. 
\end{equation*}
\end{lemma}
\begin{proof}
For a word $w \in \Cat_n$, by Lemma \ref{lem:Catxy} we can write $w=xvy$ where $v$ is a word of length $2n-2$. Write $v=a_1a_2 \cdots a_{2n-2}$. 

\medskip
\noindent By Definition \ref{def:N^m(w)}, 
\begin{equation}\label{eq:N^0(w)}
\nabla^{(0)}(w)=\prod_{i=1}^{2n-2}[1+\overline{a}_1+\overline{a}_2+\cdots+\overline{a}_i]_q. 
\end{equation}

\noindent Moreover, $\nabla^{(0)}(w)=0$ if and only if there exists an integer $i$ ($1 \leq i \leq 2n-2$) such that $\overline{a}_1+\overline{a}_2+\cdots+\overline{a}_i=-1$ if and only if $v$ is not Catalan. 

\medskip
\noindent By the above comment, the map 
\begin{align*}
\Cat_{n-1} & \to \{w \in \Cat_n \mid \nabla^{(0)}(w) \neq 0\}\\
v &\mapsto xvy
\end{align*}
is a bijection. 

\medskip
\noindent Therefore, by Definition \ref{def:N^m_n} we have 
\begin{align*}
\nabla^{(0)}_n&=\sum_{w \in \Cat_n}\nabla^{(0)}(w)w\\
&=\sum_{v \in \Cat_{n-1}}\nabla^{(0)}(xvy)xvy\\
&=xC_{n-1}y, 
\end{align*}
where the last step follows by \eqref{eq:N^0(w)} and Definition \ref{def:Cn}. 
\end{proof}

\section{Some properties of $\nabla^{(m)}(w)$}
Let $m \in \mZ$ and let $w$ denote a nontrivial Catalan word. In this section, we express $\nabla^{(m)}(w)$ in terms of the profile of $w$. 

\begin{lemma}\label{lem:N^m(w)}\rm
Let $m \in \mZ$. For $n \geq 1$ and a Catalan word $w=a_1a_2 \cdots a_{2n}$, 
\begin{equation}\label{eq:N^m(w)xy}
\nabla^{(m)}(w)=\left(\prod_{\substack{2 \leq i \leq 2n\\a_i=x}}[\overline{a}_1+\overline{a}_2+\cdots+\overline{a}_{i-1}+m]_q\right)\left(\prod_{\substack{2 \leq i \leq 2n\\a_i=y}}[\overline{a}_1+\overline{a}_2+\cdots+\overline{a}_{i-1}]_q\right). 
\end{equation}
\end{lemma}
\begin{proof}
Follows from Definition \ref{def:N^m(w)} and \eqref{eq:xy01}. 
\end{proof}

\noindent Observe that in \eqref{eq:N^m(w)xy}, the first product depends on $m$ and the second product does not depend on $m$. The following definition is motivated by this observation. 

\begin{definition}\label{def:N^m(w)xy}\rm
Let $m \in \mZ$. For $n \geq 1$ and a Catalan word $w=a_1a_2 \cdots a_{2n}$, define 
\begin{equation*}
\nabla^{(m)}_x(w)=\prod_{\substack{2 \leq i \leq 2n \\ a_i=x}}[\overline{a}_1+\overline{a}_2+\cdots+\overline{a}_{i-1}+m]_q, 
\end{equation*}
\begin{equation*}
\nabla_y(w)=\prod_{\substack{2 \leq i \leq 2n \\ a_i=y}}[\overline{a}_1+\overline{a}_2+\cdots+\overline{a}_{i-1}]_q. 
\end{equation*}
\end{definition}

\begin{lemma}\label{lem:N=NxNy}\rm
For $m \in \mZ$ and a nontrivial Catalan word $w$, 
\begin{equation*}
\nabla^{(m)}(w)=\nabla^{(m)}_x(w)\nabla_y(w). 
\end{equation*}
\end{lemma}
\begin{proof}
Follows from Lemma \ref{lem:N^m(w)} and Definition \ref{def:N^m(w)xy}. 
\end{proof}

\begin{lemma}\label{lem:Ix=Iy}\rm
Let $k \in \mZ$ and let $n \geq 1$. For a balanced word $a_1a_2 \cdots a_{2n}$, the following sets have the same cardinality: 
\begin{equation}\label{eq:Ix}
\{1 \leq i \leq 2n \mid a_i=x \text{ and } \overline{a}_1+\overline{a}_2+\cdots+\overline{a}_i=k\}, 
\end{equation}
\begin{equation}\label{eq:Iy}
\{1 \leq i \leq 2n \mid a_i=y \text{ and } \overline{a}_1+\overline{a}_2+\cdots+\overline{a}_{i-1}=k\}. 
\end{equation}
\end{lemma}
\begin{proof}
We consider the Dyck path of the word $a_1a_2 \cdots a_{2n}$. We interpret \eqref{eq:Ix} as the set of edges in the Dyck path that rise from elevation $k-1$ to elevation $k$. Similarly, we interpret \eqref{eq:Iy} as the set of edges in the Dyck path that fall from elevation $k$ to elevation $k-1$. Recall that the Dyck path of the balanced word $a_1a_2 \cdots a_{2n}$ starts and ends at elevation $0$. The result follows. 
\end{proof}

\begin{lemma}\label{lem:Nx=Ny}\rm
For a nontrivial Catalan word $w$, 
\begin{equation*}
\nabla_y(w)=\nabla^{(1)}_x(w). 
\end{equation*}
\end{lemma}
\begin{proof}
Write $w=a_1a_2 \cdots a_{2n}$ with $n \geq 1$. 

\medskip
\noindent By Definition \ref{def:N^m(w)xy} with $m=1$, we have 
\begin{equation*}
\nabla^{(1)}_x(w)=\prod_{\substack{2 \leq i \leq 2n \\ a_i=x}}[\overline{a}_1+\overline{a}_2+\cdots+\overline{a}_{i}]_q, 
\end{equation*}
\begin{equation*}
\nabla_y(w)=\prod_{\substack{2 \leq i \leq 2n \\ a_i=y}}[\overline{a}_1+\overline{a}_2+\cdots+\overline{a}_{i-1}]_q. 
\end{equation*}

\noindent Since $a_1=x$ and $\overline{x}=1$, we can write 
\begin{equation}\label{eq:Nx}
\nabla^{(1)}_x(w)=\prod_{\substack{1 \leq i \leq 2n \\ a_i=x}}[\overline{a}_1+\overline{a}_2+\cdots+\overline{a}_{i}]_q, 
\end{equation}
\begin{equation}\label{eq:Ny}
\nabla_y(w)=\prod_{\substack{1 \leq i \leq 2n \\ a_i=y}}[\overline{a}_1+\overline{a}_2+\cdots+\overline{a}_{i-1}]_q. 
\end{equation}

\noindent Comparing \eqref{eq:Nx} and \eqref{eq:Ny} using Lemma \ref{lem:Ix=Iy}, we obtain the result. 
\end{proof}

\begin{corollary}\label{cor:N^m(w)xx}\rm
For a nontrivial Catalan word $w$, 
\begin{equation*}
\nabla^{(m)}(w)=\nabla^{(m)}_x(w)\nabla^{(1)}_x(w). 
\end{equation*}
\end{corollary}
\begin{proof}
Follows from Lemmas \ref{lem:N=NxNy} and \ref{lem:Nx=Ny}. 
\end{proof}

\begin{corollary}\label{cor:N^m(w)}\rm
Let $m \in \mZ$. For $n \geq 1$ and a Catalan word $w=a_1a_2 \cdots a_{2n}$, 
\begin{equation*}
\nabla^{(m)}(w)=\prod_{\substack{2 \leq i \leq 2n \\ a_i=x}}\Big([\overline{a}_1+\overline{a}_2+\cdots+\overline{a}_{i-1}+m]_q[\overline{a}_1+\overline{a}_2+\cdots+\overline{a}_{i-1}+1]_q\Big). 
\end{equation*}
\end{corollary}
\begin{proof}
Follows from Definition \ref{def:N^m(w)xy} and Corollary \ref{cor:N^m(w)xx}. 
\end{proof}

\noindent The following definition is for notational convenience. 

\begin{definition}\label{def:qpermute}\rm
For $m,n \in \mZ$, define 
\begin{equation*}
\qpermute{m}{n}=
\begin{cases}
[m]_q[m-1]_q \cdots [m-n+1]_q, &\hspace{1em}\text{ if }n \geq 1; \\
1, &\hspace{1em}\text{ if }n=0; \\
0, &\hspace{1em}\text{ if }n \leq -1. 
\end{cases}
\end{equation*}
\end{definition}

\begin{lemma}\label{lem:N^m(w)profile}\rm
For $m \in \mZ$ and a nontrivial Catalan word $w$ with profile $(l_0,h_1,l_1,\ldots,h_r,l_r)$, 
\begin{align*}
&\nabla^{(m)}(w)\\
&=\qpermute{h_1+m-1}{h_1-l_0-1}\qpermute{h_1}{h_1-l_0-1}\qpermute{h_2+m-1}{h_2-l_1}\qpermute{h_2}{h_2-l_1} \cdots \qpermute{h_r+m-1}{h_r-l_{r-1}}\qpermute{h_r}{h_r-l_{r-1}}. 
\end{align*}
\end{lemma}
\begin{proof}
Follows from Corollary \ref{cor:N^m(w)} and Definition \ref{def:qpermute}. 
\end{proof}

\noindent Motivated by the above lemma, we make a definition. 

\begin{definition}\label{def:N^m(w)profile}\rm
For $m \in \mZ$ and a sequence $(l_0,h_1,l_1,\ldots,h_r,l_r)$ of integers with $r \geq 1$, define 
\begin{align*}
&\nabla^{(m)}(l_0,h_1,l_1,\ldots,h_r,l_r)\\
&=\qpermute{h_1+m-1}{h_1-l_0-1}\qpermute{h_1}{h_1-l_0-1}\qpermute{h_2+m-1}{h_2-l_1}\qpermute{h_2}{h_2-l_1} \cdots \qpermute{h_r+m-1}{h_r-l_{r-1}}\qpermute{h_r}{h_r-l_{r-1}}. 
\end{align*}
\end{definition}

\noindent We list four identities for $[n]_q$ that will be useful. 

\begin{lemma}\label{lem:[n]_q}\rm
The following identities hold. 
\begin{enumerate}
\item For $a,b,c \in \mZ$, 
\begin{equation*}
[a+c]_q[b+c]_q-[a]_q[b]_q=[c]_q[a+b+c]_q. 
\end{equation*}
\item For $a,b,c \in \mZ$, 
\begin{equation*}
[a]_q[b-c]_q+[b]_q[c-a]_q+[c]_q[a-b]_q=0. 
\end{equation*}
\item For $a,b,c,d \in \mZ$, 
\begin{equation*}
[a]_q[b]_q[c-d]_q+[b]_q[c]_q[d-a]_q+[c]_q[d]_q[a-b]_q+[d]_q[a]_q[b-c]_q=0. 
\end{equation*}
\item For $a,b,c,d \in \mZ$, 
\begin{equation*}
[a]_q[b]_q[a-b]_q+[b]_q[c]_q[b-c]_q+[c]_q[d]_q[c-d]_q+[d]_q[a]_q[d-a]_q=[a-c]_q[b-d]_q[a+c-b-d]_q. 
\end{equation*}
\end{enumerate}
\end{lemma}
\begin{proof}
By direct computation. 
\end{proof}

\section{A recurrence relation for $\nabla^{(m)}_n$}
Let $m \in \mZ$ and $n \geq 1$. In this section, we obtain a recurrence relation that gives $\nabla^{(m)}_{n+1}$ in terms of $\nabla^{(m)}_n$. This will be achieved in Proposition \ref{prop:N^m_nrecursion}. 

\medskip
\noindent We begin with a useful formula about nontrivial Catalan words. To avoid a degenerate situation, we exclude the Catalan word $xy$. 

\begin{lemma}\label{lem:N^m(profile)}\rm
Consider a nontrivial Catalan word other than $xy$, and let $(l_0,h_1,l_1,\ldots,h_r,l_r)$ denote its profile. Then for $m \in \mZ$ we have 
\begin{align*}
&\nabla^{(m)}(l_0,h_1,l_1,\ldots,h_r,l_r)\\
&=\sum_{j=\xi}^{r-1}\nabla^{(m)}(l_0,h_1,l_1,\ldots,h_j,l_j,h_{j+1}-1,\ldots,h_r-1,l_r)\\
&\hspace{16em} \times \Big([h_{j+1}]_q[h_{j+1}+m-1]_q-[l_j]_q[l_j+m-1]_q\Big),
\end{align*}
where $\xi=\max\{j \mid 0 \leq j \leq r-1,l_j=0\}$. 
\end{lemma}
\begin{proof}
If $\xi=r-1$, the result follows by Definition \ref{def:N^m(w)profile} and the fact that $[l_\xi]_q=0$. 

\medskip
\noindent If $\xi<r-1$, by Definition \ref{def:N^m(w)profile} and the fact that $[l_\xi]_q=0$, the right-hand side of the above equation is equal to 
\begin{align*}
&\nabla^{(m)}(l_0,h_1,l_1,\ldots,h_\xi,l_\xi,h_{\xi+1},l_{\xi+1}-1,\ldots,h_r-1,l_r)\\
&+\sum_{j=\xi+1}^{r-2}\bigg(\nabla^{(m)}(l_0,h_1,l_1,\ldots,h_j,l_j,h_{j+1},l_{j+1}-1,\ldots,h_r-1,l_r)\\
&\hspace{8em}-\nabla^{(m)}(l_0,h_1,l_1,\ldots,h_j,l_j-1,\ldots,h_r-1,l_r)\bigg)\\
&+\bigg(\nabla^{(m)}(l_0,h_1,l_1,\ldots,h_r,l_r)-\nabla^{(m)}(l_0,h_1,l_1,\ldots,h_{r-1},l_{r-1}-1,h_r-1,l_r)\bigg). 
\end{align*}
The above sum is telescoping. Upon cancellation we obtain the result. 
\end{proof}

\noindent To see why the word $xy$ is excluded from Lemma \ref{lem:N^m(profile)}, we consider the following interpretation of this lemma. Let $m \in \mZ$. For $n \geq 1$ and a word $v \in \Cat_{n+1}$, Lemma \ref{lem:N^m(profile)} writes $\nabla^{(m)}(v)$ in terms of some $\nabla^{(m)}(w)$ where $w \in \Cat_n$. To include the word $xy$, we need to extend the result to $n=0$. But this is not possible since $\nabla^{(m)}(\m1)$ is not defined. 

\begin{lemma}\label{lem:innerprod0}\rm
Let $m \in \mZ$. For $n \geq 1$ and a Catalan word $w=a_1a_2 \cdots a_{2n}$, 
\begin{equation*}
\frac{q^m x \star w-q^{-m} w \star x}{q-q^{-1}}=\sum_{i=0}^{2n}a_1 \cdots a_ixa_{i+1} \cdots a_{2n}[m+2\overline{a}_1+2\overline{a}_2+\cdots+2\overline{a}_i]_q. 
\end{equation*}
\end{lemma}
\begin{proof}
By the definition of the $q$-shuffle product in Section 2, we have 
\begin{align*}
&\frac{q^m x \star w-q^{-m} w \star x}{q-q^{-1}} \\
&\hspace{2em}=\sum_{i=0}^{2n} a_1 \cdots a_ixa_{i+1} \cdots a_{2n} \hspace{0.25em} \frac{q^{m+2\overline{a}_1+\cdots+2\overline{a}_i}-q^{-m+2\overline{a}_{i+1}+\cdots+2\overline{a}_{2n}}}{q-q^{-1}} \\
&\hspace{2em}=\sum_{i=0}^{2n} a_1 \cdots a_ixa_{i+1} \cdots a_{2n} \hspace{0.25em} \frac{q^{m+2\overline{a}_1+\cdots+2\overline{a}_i}-q^{-m-2\overline{a}_1-\cdots-2\overline{a}_i}}{q-q^{-1}} \\
&\hspace{2em}=\sum_{i=0}^{2n} a_1 \cdots a_ixa_{i+1} \cdots a_{2n}[m+2\overline{a}_1+2\overline{a}_2+\cdots+2\overline{a}_i]_q.
\end{align*}
\end{proof}

\noindent For notation convenience, we bring in a bilinear form on $\mV$. 
\begin{definition}\label{def:biform}\rm
(See \cite[Page 6]{ter_catalan}.) Let $(~,~):\mV \times \mV \to \mF$ denote the bilinear form given by $(w,w)=1$ for a word $w \in \mV$ and $(w,v)=0$ for distinct words $w,v \in \mV$. 
\end{definition}
\noindent One can rountinely check that $(~,~)$ is symmetric and nondegenerate. For a word $w \in \mV$ and any $u \in \mV$, the scalar $(w,u)$ is the coefficient of $w$ in $u$. 

\begin{lemma}\label{lem:innerprod}\rm
Let $m \in \mZ$ and let $n \geq 1$. For a word $v$ and a Catalan word $w=a_1a_2 \cdots a_{2n}$, consider the scalar 
\begin{equation}\label{eq:innerprod}
\left(\frac{(q^m x \star w-q^{-m} w \star x)y}{q-q^{-1}},v\right).
\end{equation}
\begin{enumerate}
\item If $v$ is Catalan and of length $2n+2$, then the scalar \eqref{eq:innerprod} is equal to 
	\begin{equation*}
	\sum_i[m+2\overline{a}_1+2\overline{a}_2+\cdots+2\overline{a}_i]_q, 
	\end{equation*}
	where the sum is over all $i$ $(0 \leq i \leq 2n)$ such that $v=a_1 \cdots a_ixa_{i+1} \cdots a_{2n}y$. 
\item If $v$ is not Catalan or is not of length $2n+2$, then the scalar \eqref{eq:innerprod} is equal to $0$. 
\end{enumerate}
\end{lemma}
\begin{proof}
Follows from Lemma \ref{lem:innerprod0}. 
\end{proof}

\begin{lemma}\label{lem:N^m(v)decomp}\rm
Let $m \in \mZ$. For $n \geq 1$ and a word $v \in \Cat_{n+1}$, we have 
\begin{equation*}
\nabla^{(m)}(v)=\sum_{w \in \Cat_n}\nabla^{(m)}(w)\left(\frac{(q^m x \star w-q^{-m} w \star x)y}{q-q^{-1}},v\right). 
\end{equation*}
\end{lemma}
\begin{proof}
By Lemma \ref{lem:innerprod}, it suffices to show that $\nabla^{(m)}(v)$ is equal to 
\begin{equation}\label{eq:N^m(v)}
\sum_w\sum_i\nabla^{(m)}(w)[m+2\overline{a}_1+2\overline{a}_2+\cdots+2\overline{a}_i]_q,
\end{equation}
where the first sum is over all $w=a_1a_2 \cdots a_{2n} \in \Cat_n$ such that $v=a_1 \cdots a_ixa_{i+1} \cdots a_{2n}y$ for some $i$ ($0 \leq i \leq 2n$), and the second sum is over all such $i$. 

\medskip
\noindent Let $(l_0,h_1,l_1,\ldots,h_r,l_r)$ denote the profile of $v$. Since $v$ is nontrivial, we have $r \geq 1$. Let $\xi=\max\{j \mid 0 \leq j \leq r-1,l_j=0\}$. 

\medskip
\noindent To compute the sum \eqref{eq:N^m(v)}, we consider what kind of words $w$ are being summed over, and for such $w$ what is the corresponding sum over $i$. 

\medskip
\noindent For any word $w$ being summed over in \eqref{eq:N^m(v)}, there exists an integer $j$ ($\xi \leq j \leq r-1$) such that the following hold: 
\begin{enumerate}
\item if $l_j<h_{j+1}-1$ and $j<r-1$, then the profile of $w$ is given by 
\begin{equation*}
(l_0,h_1,l_1,\ldots,h_j,l_j,h_{j+1}-1,\ldots,h_r-1,l_r); 
\end{equation*}
\item if $l_j<h_{j+1}-1$ and $j=r-1$, then the profile of $w$ is given by 
\begin{equation*}
(l_0,h_1,l_1,\ldots,h_j,l_j,h_r-1,l_r); 
\end{equation*}
\item if $l_j=h_{j+1}-1$ and $j<r-1$, then the profile of $w$ is given by 
\begin{equation*}
(l_0,h_1,l_1,\ldots,h_j,l_{j+1}-1,\ldots,h_r-1,l_r); 
\end{equation*}
\item if $l_j=h_{j+1}-1$ and $j=r-1$, then the profile of $w$ is given by 
\begin{equation*}
(l_0,h_1,l_1,\ldots,h_j,l_r). 
\end{equation*}
\end{enumerate}

\noindent For each of the cases above, by Definitions \ref{def:qpermute} and \ref{def:N^m(w)profile} we have that 
\begin{equation*}
\nabla^{(m)}(w)=\nabla^{(m)}(l_0,h_1,l_1,\ldots,h_j,l_j,h_{j+1}-1,\ldots,h_r-1,l_r). 
\end{equation*}

\medskip
\noindent For such $w$, the corresponding sum over $i$ in \eqref{eq:N^m(v)} is equal to 
\begin{align*}
&\sum_i[m+2\overline{a}_1+2\overline{a}_2+\cdots+2\overline{a}_i]_q\\
&=\sum_{s=l_j}^{h_{j+1}-1}[m+2s]_q\\
&=\sum_{s=l_j}^{h_{j+1}-1}\Big([s+1]_q[s+m]_q-[s]_q[s+m-1]_q\Big)\\
&=[h_{j+1}]_q[h_{j+1}+m-1]_q-[l_j]_q[l_j+m-1]_q, 
\end{align*}
where the second step follows by setting $a=s, b=s+m-1, c=1$ in Lemma \ref{lem:[n]_q}(i). 

\medskip
\noindent Using the above comments, we now evaluate the entire sum in \eqref{eq:N^m(v)}. We have 
\begin{align*}
&\sum_w\sum_i\nabla^{(m)}(w)[m+2\overline{a}_1+2\overline{a}_2+\cdots+2\overline{a}_i]_q\\
&=\sum_{j=\xi}^{r-1}\nabla^{(m)}(l_0,h_1,l_1,\ldots,h_j,l_j,h_{j+1}-1,\ldots,h_r-1,l_r)\\
&\hspace{8em} \times \Big([h_{j+1}]_q[h_{j+1}+m-1]_q-[l_j]_q[l_j+m-1]_q\Big)\\
&=\nabla^{(m)}(l_0,h_1,l_1,\ldots,h_r,l_r) \\
&=\nabla^{(m)}(v), 
\end{align*}
where the second step follows by Lemma \ref{lem:N^m(profile)}. 
\end{proof}

\begin{proposition}\label{prop:N^m_nrecursion}\rm
For $m \in \mZ$ and $n \geq 1$, 
\begin{equation}\label{eq:N^m_nrecursion}
\nabla^{(m)}_{n+1}=\frac{\left(q^m x \star \nabla^{(m)}_n-q^{-m} \nabla^{(m)}_n \star x\right)y}{q-q^{-1}}. 
\end{equation}
\end{proposition}
\begin{proof}
Let $v$ denote a word. We will show that the inner products of $v$ with the both sides of \eqref{eq:N^m_nrecursion} coincide. 

\medskip
\noindent If $v$ does not have length $2n+2$, then the two inner products are both $0$. 

\medskip
\noindent If $v$ is not Catalan, then the two inner products are both $0$ by Definition \ref{def:N^m_n} and Lemma \ref{lem:innerprod}. 

\medskip
\noindent If $v \in \Cat_{n+1}$, then by Definition \ref{def:N^m_n} and Lemma \ref{lem:N^m(v)decomp}, 
\begin{align*}
&\left(\frac{\left(q^m x \star \nabla^{(m)}_n-q^{-m} \nabla^{(m)}_n \star x\right)y}{q-q^{-1}},v\right)\\
&=\sum_{w \in \Cat_n}\nabla^{(m)}(w)\left(\frac{(q^m x \star w-q^{-m} w \star x)y}{q-q^{-1}},v\right)\\
&=\nabla^{(m)}(v) \\
&=\left(\nabla^{(m)}_{n+1},v\right). 
\end{align*}
\end{proof}

\noindent We consider the case $m=0$ in Proposition \ref{prop:N^m_nrecursion}. Recall that in Lemma \ref{lem:N^0_n} we have showed that $\nabla^{(0)}_n=xC_{n-1}y$ for all $n \geq 1$. 

\begin{corollary}\label{cor:xCnyrecursion}\rm
(See \cite[Lemma 7.3]{ter_beck}.) For $n \geq 1$, 
\begin{equation*}
xC_n=\frac{x \star xC_{n-1}y-xC_{n-1}y \star x}{q-q^{-1}}. 
\end{equation*}
\end{corollary}
\begin{proof}
Setting $m=0$ in Proposition \ref{prop:N^m_nrecursion} and applying Lemma \ref{lem:N^0_n} gives 
\begin{equation*}
xC_ny=\frac{\left(x \star xC_{n-1}y-xC_{n-1}y \star x\right)y}{q-q^{-1}}. 
\end{equation*}
In the above equation, on each side remove the $y$ on the right to obtain the result. 
\end{proof}

\noindent Next we will show that the elements $\{\nabla^{(m)}_n\}_{m \in \mZ,n \geq 1}$ mutually commute with repect to the $q$-shuffle product. This will be achieved in three steps. Before giving the steps, we have some comments. Recall that above Corollary \ref{cor:xCnyrecursion} we mentioned that $\nabla^{(0)}_n=xC_{n-1}y$ for all $n \geq 1$. Also, by Lemma \ref{lem:N^m_1} we have $\nabla^{(m)}_1=xy$ for all $m \in \mZ$. We now give the three steps. 
\begin{enumerate}
\item In Section 8, we will show that $\nabla^{(m)}_n$ commutes with $xy$ for $m \in \mZ$ and $n \geq 1$. 
\item In Section 9, we will show that the elements $\{\nabla^{(0)}_n\}_{n \geq 1}$ mutually commute. 
\item In Section 11, we will show that the elements $\{\nabla^{(m)}_n\}_{m \in \mZ,n \geq 1}$ mutually commute. 
\end{enumerate}

\section{$\nabla^{(m)}_n$ commutes with $xy$}
Let $m \in \mZ$ and let $n \geq 1$. In this section, we show that $\nabla^{(m)}_n$ commutes with $xy$ with respect to the $q$-shuffle product. 

\medskip
\noindent The following results will be useful. 

\begin{lemma}\label{lem:catprod}\rm
Let $n,k \in \mN$. If $v \in \Cat_n$ and $w \in \Cat_k$, then $v \star w$ is a linear combination of words in $\Cat_{n+k}$. 
\end{lemma}
\begin{proof}
Follows from the definition of the $q$-shuffle product in Section 2. 
\end{proof}

\begin{lemma}\label{lem:l_j=-1}\rm
Let $m \in \mZ$ and let $(l_0,h_1,l_1,\ldots,h_r,l_r)$ denote a sequence of integers with $r \geq 1$. If there exists an integer $i$ ($1 \leq i \leq r-1$) such that $l_i=-1 \neq h_{i+1}$, then 
\begin{equation*}
\nabla^{(m)}(l_0,h_1,l_1,\ldots,h_r,l_r)=0. 
\end{equation*}
\end{lemma}
\begin{proof}
Follows from Definitions \ref{def:qpermute} and \ref{def:N^m(w)profile}. 
\end{proof}

\begin{lemma}\label{lem:coeff[]v1}\rm
Let $m \in \mZ$. For $n \geq 1$ and $w \in \Cat_{n+1}$ with profile $(l_0,h_1,l_1,\ldots,h_r,l_r)$, the coefficient of $w$ in $\dfrac{xy \star \nabla^{(m)}_n-\nabla^{(m)}_n \star xy}{q-q^{-1}}$ is equal to 
\begin{align*}
&\sum_{0 \leq i<j \leq r}\nabla^{(m)}(l_0,h_1,l_1,\ldots,l_i,h_{i+1}-1,\ldots,h_j-1,l_j,\ldots,h_r,l_r)\\
&\hspace{16em} \times [l_i-h_{i+1}]_q[l_j-h_j]_q[l_i+h_{i+1}-l_j-h_j]_q. 
\end{align*}
\end{lemma}
\begin{proof}
To prove our result, we consider what are the words $v \in \Cat_n$ such that $(xy \star v,w)$ or $(v \star xy,w)$ is nonzero, and for such $v$ what is the coefficient of $w$ in $\dfrac{xy \star v-v \star xy}{q-q^{-1}}$. 

\medskip
\noindent For any $v \in \Cat_n$ such that $(xy \star v,w)$ or $(v \star xy,w)$ is nonzero, let us compare the profile of $v$ with the profile of $w$. By construction, there exists integers $i,j$ ($0 \leq i<j \leq r$) such that $l_k \geq 1$ for $i<k<j$ and the following hold: 
\begin{enumerate}[(i)]
\item if $l_i<h_{i+1}-1$, $h_j-1>l_j$, and $i<j-1$, then the profile of $v$ is given by 
\begin{equation*}
(l_0,h_1,l_1,\ldots,l_i,h_{i+1}-1,\ldots,h_j-1,l_j,\ldots,h_r,l_r); 
\end{equation*}
\item if $l_i<h_{i+1}-1$, $h_j-1>l_j$, and $i=j-1$, then the profile of $v$ is given by 
\begin{equation*}
(l_0,h_1,l_1,\ldots,l_i,h_{i+1}-1,l_j,\ldots,h_r,l_r); 
\end{equation*}
\item if $l_i=h_{i+1}-1$, $h_j-1>l_j$, and $i<j-1$, then the profile of $v$ is given by 
\begin{equation*}
(l_0,h_1,l_1,\ldots,h_i,l_{i+1}-1,\ldots,h_j-1,l_j,\ldots,h_r,l_r); 
\end{equation*}
\item if $l_i=h_{i+1}-1$, $h_j-1>l_j$, and $i=j-1$, then the profile of $v$ is given by 
\begin{equation*}
(l_0,h_1,l_1,\ldots,h_i,l_j,\ldots,h_r,l_r); 
\end{equation*}
\item if $l_i<h_{i+1}-1$, $h_j-1=l_j$ and $i<j-1$, then the profile of $v$ is given by 
\begin{equation*}
(l_0,h_1,l_1,\ldots,l_i,h_{i+1}-1,\ldots,l_{j-1}-1,h_{j+1},\ldots,h_r,l_r); 
\end{equation*}
\item if $l_i<h_{i+1}-1$, $h_j-1=l_j$ and $i=j-1$, then the profile of $v$ is given by 
\begin{equation*}
(l_0,h_1,l_1,\ldots,l_i,h_{j+1},\ldots,h_r,l_r); 
\end{equation*}
\item if $l_i=h_{i+1}-1$, $h_j-1=l_j$, and $i<j-1$, then the profile of $v$ is given by 
\begin{equation*}
(l_0,h_1,l_1,\ldots,h_i,l_{i+1}-1,\ldots,l_{j-1}-1,h_{j+1},\ldots,h_r,l_r); 
\end{equation*}
\item if $l_i=h_{i+1}-1$, $h_j-1=l_j$, and $i=j-1$, then the profile of $v$ is given by 
\begin{equation*}
(l_0,h_1,l_1,\ldots,h_i,l_i,h_{j+1},\ldots,h_r,l_r). 
\end{equation*}
\end{enumerate}

\medskip
\noindent For each of the cases above, by Definitions \ref{def:qpermute} and \ref{def:N^m(w)profile} we have that 
\begin{equation*}
\nabla^{(m)}(v)=\nabla^{(m)}(l_0,h_1,l_1,\ldots,l_i,h_{i+1}-1,\ldots,h_j-1,l_j,\ldots,h_r,l_r). 
\end{equation*}

\medskip
\noindent For such $v$, we have that 
\begin{equation*}
\left(\frac{xy \star v-v \star xy}{q-q^{-1}},w\right)=\sum_{i,j}[l_i-h_{i+1}]_q[l_j-h_j]_q[l_i+h_{i+1}-l_j-h_j]_q, 
\end{equation*}
where the sum is over all $i,j$ such that the profile of $v$ is given in one of the cases (i)-(viii) above. 

\medskip
\noindent By the above comments, the coefficient of $w$ in $\dfrac{xy \star \nabla^{(m)}_n-\nabla^{(m)}_n \star xy}{q-q^{-1}}$ is equal to 
\begin{align*}
&\sum_{\substack{0 \leq i<j \leq r \\ l_k \geq 1 \text{ for }i<k<j}}\nabla^{(m)}(l_0,h_1,l_1,\ldots,l_i,h_{i+1}-1,\ldots,h_j-1,l_j,\ldots,h_r,l_r)\\
&\hspace{16em} \times [l_i-h_{i+1}]_q[l_j-h_j]_q[l_i+h_{i+1}-l_j-h_j]_q. 
\end{align*}

\noindent By Lemma \ref{lem:l_j=-1} the condition 
\begin{equation*}
l_k \geq 1 \text{ for } i<k<j 
\end{equation*}
in the above sum can be dropped. The result follows. 
\end{proof}

\begin{lemma}\label{lem:coeff[]v2}\rm
Let $m \in \mZ$. For $n \geq 1$ and $w \in \Cat_{n+1}$ with profile $(l_0,h_1,l_1,\ldots,h_r,l_r)$, the coefficient of $w$ in $\dfrac{xy \star \nabla^{(m)}_n-\nabla^{(m)}_n \star xy}{q-q^{-1}}$ is equal to 
\begin{align*}
&\sum_{1 \leq i<j \leq r-1}\nabla^{(m)}(l_0,h_1,l_1,\ldots,l_{i-1},h_i-1,l_i,h_{i+1}-1,\ldots,h_j-1,l_j,h_{j+1}-1,l_{j+1},\ldots,h_r,l_r)\\
&\hspace{4em} \times [l_i]_q[l_j]_q[h_i]_q[h_{j+1}]_q\\
&\hspace{4em} \times \Big([h_i+m-1]_q[h_{j+1}+m-1]_q[l_i-l_j]_q+[h_{j+1}+m-1]_q[l_i+m-1]_q[l_j-h_i]_q\\
&\hspace{7em}+[l_i+m-1]_q[l_j+m-1]_q[h_i-h_{j+1}]_q+[l_j+m-1]_q[h_i+m-1]_q[h_{j+1}-l_i]_q\Big)\\
&+\sum_{i=1}^{r-1}\nabla^{(m)}(l_0,h_1,l_1,\ldots,l_{i-1},h_i-1,l_i,h_{i+1}-1,l_{i+1},\ldots,h_r,l_r)\\
&\hspace{4em} \times [l_i]_q[h_i]_q[h_{i+1}]_q\\
&\hspace{4em} \times \Big([h_{i+1}+m-1]_q[l_i-h_i]_q+[l_i+m-1]_q[h_i-h_{i+1}]_q+[h_i+m-1]_q[h_{i+1}-l_i]_q\Big). 
\end{align*}
\end{lemma}
\begin{proof}
By setting $a=l_i, b=l_j, c=h_{i+1}, d=h_j$ in Lemma \ref{lem:[n]_q}(iv), we have 
\begin{align*}
&[l_i-h_{i+1}]_q[l_j-h_j]_q[l_i+h_{i+1}-l_j-h_j]_q\\
&=[l_i]_q[l_j]_q[l_i-l_j]_q+[l_j]_q[h_{i+1}]_q[l_j-h_{i+1}]_q+[h_{i+1}]_q[h_j]_q[h_{i+1}-h_j]_q+[h_j]_q[l_i]_q[h_j-l_i]_q. 
\end{align*}

\noindent Applying the above equation to Lemma \ref{lem:coeff[]v1} and shifting indices (note that the boundary terms that come up are always equal to $0$) shows that the coefficient of $w$ in $\dfrac{xy \star \nabla^{(m)}_n-\nabla^{(m)}_n \star xy}{q-q^{-1}}$ is equal to 
\begin{align}
&\sum_{1 \leq i<j \leq r-1}\nabla^{(m)}(l_0,h_1,l_1,\ldots,l_i,h_{i+1}-1,\ldots,h_j-1,l_j,\ldots,h_r,l_r)[l_i]_q[l_j]_q[l_i-l_j]_q\label{eq:foursums1}\\
&+\sum_{1 \leq i \leq j \leq r-1}\nabla^{(m)}(l_0,h_1,l_1,\ldots,l_{i-1},h_i-1,\ldots,h_j-1,l_j,\ldots,h_r,l_r)[l_j]_q[h_i]_q[l_j-h_i]_q\label{eq:foursums2}\\
&+\sum_{1 \leq i \leq j \leq r-1}\nabla^{(m)}(l_0,h_1,l_1,\ldots,l_{i-1},h_i-1,\ldots,h_{j+1}-1,l_{j+1},\ldots,h_r,l_r)[h_i]_q[h_{j+1}]_q[h_i-h_{j+1}]_q\label{eq:foursums3}\\
&+\sum_{1 \leq i \leq j \leq r-1}\nabla^{(m)}(l_0,h_1,l_1,\ldots,l_i,h_{i+1}-1,\ldots,h_{j+1}-1,l_{j+1},\ldots,h_r,l_r)[h_{j+1}]_q[l_i]_q[h_{j+1}-l_i]_q.\label{eq:foursums4}
\end{align}

\noindent For $1 \leq i<j \leq r-1$ we examine the $(i,j)$-summand in the sums \eqref{eq:foursums1}--\eqref{eq:foursums4}. We will express each of these summands in terms of 
\begin{equation}\label{eq:nablai<j}
\nabla^{(m)}(l_0,h_1,l_1,\ldots,l_{i-1},h_i-1,l_i,h_{i+1}-1,\ldots,h_j-1,l_j,h_{j+1}-1,l_{j+1},\ldots,h_r,l_r). 
\end{equation}

\noindent Using Definition \ref{def:N^m(w)profile} we obtain: 
\begin{itemize}
\item the $(i,j)$-summand in \eqref{eq:foursums1} is equal to \eqref{eq:nablai<j} times 
\begin{equation*}
[l_i]_q[l_j]_q[h_i]_q[h_{j+1}]_q[h_i+m-1]_q[h_{j+1}+m-1]_q[l_i-l_j]_q;
\end{equation*}
\item the $(i,j)$-summand in \eqref{eq:foursums2} is equal to \eqref{eq:nablai<j} times 
\begin{equation*}
[l_i]_q[l_j]_q[h_i]_q[h_{j+1}]_q[h_{j+1}+m-1]_q[l_i+m-1]_q[l_j-h_i]_q;
\end{equation*}
\item the $(i,j)$-summand in \eqref{eq:foursums3} is equal to \eqref{eq:nablai<j} times 
\begin{equation*}
[l_i]_q[l_j]_q[h_i]_q[h_{j+1}]_q[l_i+m-1]_q[l_j+m-1]_q[h_i-h_{j+1}]_q;
\end{equation*}
\item the $(i,j)$-summand in \eqref{eq:foursums4} is equal to \eqref{eq:nablai<j} times 
\begin{equation*}
[l_i]_q[l_j]_q[h_i]_q[h_{j+1}]_q[l_j+m-1]_q[h_i+m-1]_q[h_{j+1}-l_i]_q. 
\end{equation*}
\end{itemize}
By these comments, For $1 \leq i<j \leq r-1$ the combined $(i,j)$-summand in the sums \eqref{eq:foursums1}--\eqref{eq:foursums4} is equal to the $(i,j)$-summand in first sum of the lemma statement. 

\medskip
\noindent Next, for $1 \leq i=j \leq r-1$ we examine the $(i,j)$-summand in the sums \eqref{eq:foursums2}--\eqref{eq:foursums4}. We will express each of these summands in terms of 
\begin{equation}\label{eq:nablai=j}
\nabla^{(m)}(l_0,h_1,l_1,\ldots,l_{i-1},h_i-1,l_i,h_{i+1}-1,l_{i+1},\ldots,h_r,l_r). 
\end{equation}

\noindent Using Definition \ref{def:N^m(w)profile} we obtain: 
\begin{itemize}
\item the $(i,j)$-summand in \eqref{eq:foursums2} is equal to \eqref{eq:nablai=j} times 
\begin{equation*}
[l_i]_q[h_i]_q[h_{i+1}]_q[h_{i+1}+m-1]_q[l_i-h_i]_q;
\end{equation*}
\item the $(i,j)$-summand in \eqref{eq:foursums3} is equal to \eqref{eq:nablai=j} times 
\begin{equation*}
[l_i]_q[h_i]_q[h_{i+1}]_q[l_i+m-1]_q[h_i-h_{i+1}]_q;
\end{equation*}
\item the $(i,j)$-summand in \eqref{eq:foursums4} is equal to \eqref{eq:nablai=j} times 
\begin{equation*}
[l_i]_q[h_i]_q[h_{i+1}]_q[h_i+m-1]_q[h_{i+1}-l_i]_q. 
\end{equation*}
\end{itemize}
By these comments, for $1 \leq i=j \leq r-1$ the combined $(i,j)$-summand in the sums \eqref{eq:foursums2}--\eqref{eq:foursums4} is equal to the $i$-summand in the second sum of the lemma statement. 
\end{proof}

\begin{proposition}\label{prop:xycommute}\rm
For $m \in \mZ$ and $n \geq 1$, 
\begin{equation*}
xy \star \nabla^{(m)}_n=\nabla^{(m)}_n \star xy. 
\end{equation*}
\end{proposition}
\begin{proof}
Referring to either of the sums in Lemma \ref{lem:coeff[]v2}, we will show that for each summand the big parenthesis is equal to $0$. 

\medskip
\noindent For $1 \leq i<j \leq r-1$, we set $a=h_i+m-1, b=h_{j+1}+m-1, c=l_i+m-1, d=l_j+m-1$ in Lemma \ref{lem:[n]_q}(iii). This yields 
\begin{align*}
&[h_i+m-1]_q[h_{j+1}+m-1]_q[l_i-l_j]_q+[h_{j+1}+m-1]_q[l_i+m-1]_q[l_j-h_i]_q\\
&\hspace{4em}+[l_i+m-1]_q[l_j+m-1]_q[h_i-h_{j+1}]_q+[l_j+m-1]_q[h_i+m-1]_q[h_{j+1}-l_i]_q\\
&=0. 
\end{align*}

\noindent For $1 \leq i \leq r-1$, we set $a=h_{i+1}+m-1, b=l_i+m-1, c=h_i+m-1$ in Lemma \ref{lem:[n]_q}(ii). This yields 
\begin{equation*}
[h_{i+1}+m-1]_q[l_i-h_i]_q+[l_i+m-1]_q[h_i-h_{i+1}]_q+[h_i+m-1]_q[h_{i+1}-l_i]_q=0. 
\end{equation*}

\noindent By Lemma \ref{lem:catprod}, we have that $xy \star \nabla^{(m)}_n-\nabla^{(m)}_n \star xy$ is a linear combination of words in $\Cat_{n+1}$. The result follows. 
\end{proof}

\begin{corollary}\label{cor:xycommutexCny}\rm
For $n \geq 1$, 
\begin{equation*}
xy \star xC_{n-1}y=xC_{n-1}y \star xy. 
\end{equation*}
\end{corollary}
\begin{proof}
Setting $m=0$ in Proposition \ref{prop:xycommute} and applying Lemma \ref{lem:N^0_n}, we obtain the result. 
\end{proof}

\section{The elements $\{\nabla^{(0)}_n\}_{n \geq 1}$ mutually commute}
In this section, we show that the elements $\{\nabla^{(0)}_n\}_{n \geq 1}$ mutually commute with respect to the $q$-shuffle product. Recall from Lemma \ref{lem:N^0_n} that $\nabla^{(0)}_n=xC_{n-1}y$ for all $n \geq 1$.

\medskip
\noindent The following definition is for notational convenience. 

\begin{definition}\label{def:y^-1}\rm
(See \cite[Lemma 4.3]{PT}.) For $n \geq 1$ and a word $w=a_1a_2 \cdots a_n$, define 
\begin{equation*}
wy^{-1}=
\begin{cases}
0,&\hspace{1em}\text{ if }a_n=x;\\
a_1a_2 \cdots a_{n-1},&\hspace{1em}\text{ if }a_n=y.
\end{cases}
\end{equation*}
\noindent We further define $\m1y^{-1}=0$. For $v \in \mV$, we define $vy^{-1}$ in a linear way. 
\end{definition}

\begin{example}\rm
We have 
\begin{equation*}
(xxyy-6xyxy+2xyyx+3yxxy-5yxyx-4yyxx)y^{-1}=xxy-6xyx+3yxx. 
\end{equation*}
\end{example}

\noindent We make the convention that the $y^{-1}$ notation has higher priority than the $q$-shuffle product. 

\begin{example}\rm
We have 
\begin{equation*}
xy \star xyxyy^{-1}=xy \star xyx, \hspace{2em}xyy^{-1} \star xyxy=x \star xyxy. 
\end{equation*}
\end{example}

\noindent The following result about $y^{-1}$ will be useful. 

\begin{lemma}\label{lem:y^-1}\rm
(See \cite[Lemma 8.3]{PT}.) Let $v,w$ be balanced words. Then 
\begin{equation*}
(v \star w)y^{-1}=vy^{-1} \star w+v \star wy^{-1}. 
\end{equation*}
\end{lemma}
\begin{proof}
Follows from the definition of the $q$-shuffle product in Section 2. 
\end{proof}

\begin{lemma}\label{lem:xycommutecor}\rm
For $ m \in \mZ$ and $n \geq 1$, 
\begin{equation*}
x \star \nabla^{(m)}_n-\nabla^{(m)}_n \star x=\nabla^{(m)}_ny^{-1} \star xy-xy \star \nabla^{(m)}_ny^{-1}. 
\end{equation*}
\end{lemma}
\begin{proof}
In the equation of Proposition \ref{prop:xycommute}, on both sides apply $y^{-1}$ on the right and evaluate the result using Lemma \ref{lem:y^-1}. This yields 
\begin{equation*}
x \star \nabla^{(m)}_n+xy \star \nabla^{(m)}_ny^{-1}=\nabla^{(m)}_ny^{-1} \star xy+\nabla^{(m)}_n \star x. 
\end{equation*}
In this equation, we rearrange the terms to obtain the result. 
\end{proof}

\begin{lemma}\label{lem:N^0_nrecursion}\rm
(See \cite[Proposition 2.20]{ter_catalan}.) For $n \geq 1$, 
\begin{enumerate}
\item $\displaystyle \nabla^{(0)}_{n+1}y^{-1}=\frac{x \star \nabla^{(0)}_n-\nabla^{(0)}_n \star x}{q-q^{-1}};$
\item $\displaystyle \nabla^{(0)}_{n+1}y^{-1}=\frac{\nabla^{(0)}_ny^{-1} \star xy-xy \star \nabla^{(0)}_ny^{-1}}{q-q^{-1}}.$
\end{enumerate}
\end{lemma}
\begin{proof}
In \eqref{eq:N^m_nrecursion}, set $m=0$ and on both sides apply $y^{-1}$ on the right. This yields (i). Evaluating (i) using Lemma \ref{lem:xycommutecor}, we obtain (ii). 
\end{proof}

\begin{lemma}\label{lem:N^0_n+ky^-1}\rm
(See \cite[Corollary 8.5]{ter_beck}.) For $n,k \geq 1$, 
\begin{equation}\label{eq:N^0_n+ky^-1}
\nabla^{(0)}_{n+k}y^{-1}=\frac{\nabla^{(0)}_ny^{-1} \star \nabla^{(0)}_k-\nabla^{(0)}_k \star \nabla^{(0)}_ny^{-1}}{q-q^{-1}}. 
\end{equation}
\end{lemma}
\begin{proof}
We use induction on $n$. 

\medskip
\noindent First assume $n=1$. In Lemma \ref{lem:N^0_nrecursion}(i), substitute $n$ with $k$ and evaluate the result using Lemma \ref{lem:N^m_1}. This yields \eqref{eq:N^0_n+ky^-1}. 

\medskip
\noindent Now assume $n \geq 2$. By induction, 
\begin{equation}\label{ih:N^0_n+ky^-1}
\nabla^{(0)}_{n+k-1}y^{-1}=\frac{\nabla^{(0)}_{n-1}y^{-1} \star \nabla^{(0)}_k-\nabla^{(0)}_k \star \nabla^{(0)}_{n-1}y^{-1}}{q-q^{-1}}. 
\end{equation}

\noindent Using in order Lemma \ref{lem:N^0_nrecursion}(ii), \eqref{ih:N^0_n+ky^-1}, Proposition \ref{prop:xycommute}, and Lemma \ref{lem:N^0_nrecursion}(ii), we have 
\begin{align*}
&\nabla^{(0)}_{n+k}y^{-1}\\
&=\frac{\nabla^{(0)}_{n+k-1}y^{-1} \star xy-xy \star \nabla^{(0)}_{n+k-1}y^{-1}}{q-q^{-1}}\\
&=\frac{\left(\nabla^{(0)}_{n-1}y^{-1} \star \nabla^{(0)}_k-\nabla^{(0)}_k \star \nabla^{(0)}_{n-1}y^{-1}\right) \star xy-xy \star \left(\nabla^{(0)}_{n-1}y^{-1} \star \nabla^{(0)}_k-\nabla^{(0)}_k \star \nabla^{(0)}_{n-1}y^{-1}\right)}{(q-q^{-1})^2}\\
&=\frac{\left(\nabla^{(0)}_{n-1}y^{-1} \star xy-xy \star \nabla^{(0)}_{n-1}y^{-1}\right) \star \nabla^{(0)}_k-\nabla^{(0)}_k \star \left(\nabla^{(0)}_{n-1}y^{-1} \star xy-xy \star \nabla^{(0)}_{n-1}y^{-1}\right)}{(q-q^{-1})^2}\\
&=\frac{\nabla^{(0)}_ny^{-1} \star \nabla^{(0)}_k-\nabla^{(0)}_k \star \nabla^{(0)}_ny^{-1}}{q-q^{-1}}. 
\end{align*}
\end{proof}

\begin{lemma}\label{lem:N^0_ncommute}\rm
(See \cite[Corollary 8.4]{ter_beck}.) For $n,k \geq 1$, 
\begin{equation*}
\nabla^{(0)}_n \star \nabla^{(0)}_k=\nabla^{(0)}_k \star \nabla^{(0)}_n. 
\end{equation*}
\end{lemma}
\begin{proof}
Swap $n$ and $k$ in \eqref{eq:N^0_n+ky^-1} and compare the result with \eqref{eq:N^0_n+ky^-1}. This yields 
\begin{equation*}
\nabla^{(0)}_ny^{-1} \star \nabla^{(0)}_k-\nabla^{(0)}_k \star \nabla^{(0)}_ny^{-1}=\nabla^{(0)}_ky^{-1} \star \nabla^{(0)}_n-\nabla^{(0)}_n \star \nabla^{(0)}_ky^{-1}. 
\end{equation*}
Rearranging the terms, we have 
\begin{equation*}
\nabla^{(0)}_ny^{-1} \star \nabla^{(0)}_k+\nabla^{(0)}_n \star \nabla^{(0)}_ky^{-1}=\nabla^{(0)}_ky^{-1} \star \nabla^{(0)}_n+\nabla^{(0)}_k \star \nabla^{(0)}_ny^{-1}. 
\end{equation*}
Evaluating the above equation using Lemma \ref{lem:y^-1}, we have 
\begin{equation*}
\left(\nabla^{(0)}_n \star \nabla^{(0)}_k\right)y^{-1}=\left(\nabla^{(0)}_k \star \nabla^{(0)}_n\right)y^{-1}. 
\end{equation*}
Using Definition \ref{def:N^m_n}, we expand out $\nabla^{(0)}_n \star \nabla^{(0)}_k$ and $\nabla^{(0)}_k \star \nabla^{(0)}_n$ and express them as linear combinations of words. Each word is Catalan by Lemma \ref{lem:catprod}. Each word ends with $y$ by Lemma \ref{lem:Catxy}. The result follows. 
\end{proof}

\section{A relation involving $\{\Delta^{(m)}_n\}_{m \in \mZ,n \in \mN}$ and $\{\nabla^{(0)}_n\}_{n \geq 1}$}
In this section, we return our attention to the elements $\{\Delta^{(m)}_n\}_{m \in \mZ,n \in \mN}$ introduced in Section 4. We will obtain a relation involving these elements and the elements $\{\nabla^{(0)}_n\}_{n \geq 1}$. This relation will be used in the next section. 

\medskip
\noindent To begin, we use Lemma \ref{lem:D&N} to reformulate Propositions \ref{prop:N^m_nrecursion}, \ref{prop:xycommute} and Lemma \ref{lem:xycommutecor} in terms of $\{\Delta^{(m)}_n\}_{m \in \mZ,n \in \mN}$. 

\begin{proposition}\label{prop:D^m_nrecursion}\rm
For $m \in \mZ$ and $n \in \mN$, 
\begin{equation}\label{eq:D^m_nrecursion}
\Delta^{(m)}_{n+1}=\frac{\left(q^m x \star \Delta^{(m)}_n-q^{-m} \Delta^{(m)}_n \star x\right)y}{q-q^{-1}}. 
\end{equation}
\end{proposition}
\begin{proof}
The case $n=0$ is routine, and the case $n \geq 1$ follows from Lemma \ref{lem:D&N} and Proposition \ref{prop:N^m_nrecursion}. 
\end{proof}

\begin{proposition}\label{prop:xycommuteD}\rm
For $m \in \mZ$ and $n \in \mN$, 
\begin{equation*}
xy \star \Delta^{(m)}_n=\Delta^{(m)}_n \star xy. 
\end{equation*}
\end{proposition}
\begin{proof}
The case $n=0$ is routine, and the case $n \geq 1$ follows from Lemma \ref{lem:D&N} and Proposition \ref{prop:xycommute}. 
\end{proof}

\begin{lemma}\label{lem:xycommutecorD}\rm
For $m \in \mZ$ and $n \in \mN$, 
\begin{equation*}
x \star \Delta^{(m)}_n-\Delta^{(m)}_n \star x=\Delta^{(m)}_ny^{-1} \star xy-xy \star \Delta^{(m)}_ny^{-1}. 
\end{equation*}
\end{lemma}
\begin{proof}
The case $n=0$ is routine, and the case $n \geq 1$ follows from Lemmas \ref{lem:D&N} and \ref{lem:xycommutecor}. 
\end{proof}

\begin{lemma}\label{lem:D^m_ny^-1}\rm
For $m \in \mZ$ and $n \in \mN$, 
\begin{enumerate}
\item $\displaystyle \Delta^{(m)}_{n+1}y^{-1}=[m]_q\sum_{k=0}^n q^{-mk} \nabla^{(0)}_{k+1}y^{-1} \star \Delta^{(m)}_{n-k};$
\item $\displaystyle \Delta^{(m)}_{n+1}y^{-1}=[m]_q\sum_{k=0}^n q^{mk} \Delta^{(m)}_{n-k} \star \nabla^{(0)}_{k+1}y^{-1}.$
\end{enumerate}
\end{lemma}
\begin{proof}
(i) We use induction on $n$. 

\medskip
\noindent The case $n=0$ is routine. 

\medskip
\noindent Now assume $n \geq 1$. By induction, 
\begin{equation}\label{ih:D^m_ny^-1}
\Delta^{(m)}_ny^{-1}=[m]_q\sum_{k=0}^{n-1} q^{-mk} \nabla^{(0)}_{k+1}y^{-1} \star \Delta^{(m)}_{n-1-k}. 
\end{equation}

\noindent By Lemma \ref{lem:N^0_nrecursion}(ii) and Proposition \ref{prop:xycommuteD}, 
\begin{align*}
&[m]_q\sum_{k=0}^n q^{-mk} \nabla^{(0)}_{k+1}y^{-1} \star \Delta^{(m)}_{n-k}\\
&=[m]_qx \star \Delta^{(m)}_n+[m]_q\sum_{k=1}^n q^{-mk} \nabla^{(0)}_{k+1}y^{-1} \star \Delta^{(m)}_{n-k}\\
&=[m]_qx \star \Delta^{(m)}_n+\frac{[m]_q}{q-q^{-1}}\sum_{k=1}^n q^{-mk} \left(\nabla^{(0)}_ky^{-1} \star \Delta^{(m)}_{n-k} \star xy-xy \star \nabla^{(0)}_ky^{-1} \star \Delta^{(m)}_{n-k}\right)\\
&=[m]_qx \star \Delta^{(m)}_n+\frac{q^{-m}[m]_q}{q-q^{-1}}\sum_{k=0}^{n-1} q^{-mk} \left(\nabla^{(0)}_{k+1}y^{-1} \star \Delta^{(m)}_{n-k-1} \star xy-xy \star \nabla^{(0)}_{k+1}y^{-1} \star \Delta^{(m)}_{n-k-1}\right)\\
&=[m]_qx \star \Delta^{(m)}_n+\frac{q^{-m}}{q-q^{-1}} \left(\Delta^{(m)}_ny^{-1} \star xy-xy \star \Delta^{(m)}_ny^{-1}\right), 
\end{align*}
where the last step follows by \eqref{ih:D^m_ny^-1}. 

\noindent Applying Lemma \ref{lem:xycommutecorD} to the above result, we have 
\begin{align*}
&[m]_q\sum_{k=0}^n q^{-mk} \nabla^{(0)}_{k+1}y^{-1} \star \Delta^{(m)}_{n-k}\\
&=[m]_qx \star \Delta^{(m)}_n+\frac{q^{-m}}{q-q^{-1}} \big(x \star \Delta^{(m)}_n-\Delta^{(m)}_n \star x\big)\\
&=\frac{q^m x \star \Delta^{(m)}_n-q^{-m} \Delta^{(m)}_n \star x}{q-q^{-1}}\\
&=\Delta^{(m)}_{n+1}y^{-1}, 
\end{align*}
where the last step follows by \eqref{eq:D^m_nrecursion}. 

\medskip
\noindent (ii) Similar to the proof of (i). 
\end{proof}

\noindent The following definition is for notational convenience. 

\begin{definition}\label{def:genfunDN}\rm
For $m \in \mZ$, define the generating function 
\begin{equation*}
\Delta^{(m)}(t)=\sum_{n \in \mN}\Delta^{(m)}_nt^n. 
\end{equation*}
We also define the generating function 
\begin{equation*}
\nabla^{(0)}(t)=\sum_{n=1}^\infty\nabla^{(0)}_nt^n. 
\end{equation*}
\end{definition}

\begin{proposition}\label{prop:D^m_-(t)}\rm
For $m \in \mZ$, 
\begin{enumerate}
\item $\Delta^{(m)}(t)y^{-1}=q^m[m]_q \nabla^{(0)}(q^{-m}t)y^{-1} \star \Delta^{(m)}(t);$
\item $\Delta^{(m)}(t)y^{-1}=q^{-m}[m]_q \Delta^{(m)}(t) \star \nabla^{(0)}(q^mt)y^{-1}.$
\end{enumerate}
\end{proposition}
\begin{proof}
This is Lemma \ref{lem:D^m_ny^-1} expressed in terms of generating functions. 
\end{proof}
\noindent Referring to Proposition \ref{prop:D^m_-(t)}, part (ii) will not be used later in the paper. It is included for the sake of completeness. 

\section{An exponential formula}
Let $m \in \mZ$. In this section, we first obtain an exponential formula relating $\Delta^{(m)}(t)$ and $\nabla^{(0)}(t)$. This formula is shown in Theorem \ref{thm:D=exp}. Using this formula, we obtain multiple corollaries that confirm our main results stated in Section 2. 

\begin{lemma}\label{lem:genfunrecursion1}\rm
We have 
\begin{equation*}
\nabla^{(0)}(t)y^{-1}=tx+\frac{tx \star \nabla^{(0)}(t)-\nabla^{(0)}(t) \star tx}{q-q^{-1}}. 
\end{equation*}
\end{lemma}
\begin{proof}
This is Lemma \ref{lem:N^0_nrecursion}(i) expressed in terms of generating functions. 
\end{proof}

\noindent We have a comment on Proposition \ref{prop:D^m_nrecursion}. In \eqref{eq:D^m_nrecursion}, on both sides apply $y^{-1}$ on the right. This shows that for $m \in \mZ$ and $n \in \mN$,  
\begin{equation}\label{eq:D^m_ny^-1recursion}
\Delta^{(m)}_{n+1}y^{-1}=\frac{q^m x \star \Delta^{(m)}_n-q^{-m} \Delta^{(m)}_n \star x}{q-q^{-1}}. 
\end{equation}

\begin{lemma}\label{lem:genfunrecursion2}\rm
For $m \in \mZ$, 
\begin{equation*}
\Delta^{(m)}(t)y^{-1}=\frac{q^m tx \star \Delta^{(m)}(t)-q^{-m} \Delta^{(m)}(t) \star tx}{q-q^{-1}}. 
\end{equation*}
\end{lemma}
\begin{proof}
This is \eqref{eq:D^m_ny^-1recursion} expressed in terms of generating functions. 
\end{proof}

\begin{lemma}\label{lem:dD/dtlem1}\rm
For $m \in \mZ$, 
\begin{equation*}
\left(\frac{d}{dt}\Delta^{(m)}(t)\right)y^{-1}=t^{-1}\Delta^{(m)}(t)y^{-1}+\frac{q^m tx \star \left(\frac{d}{dt}\Delta^{(m)}(t)\right)-q^{-m} \left(\frac{d}{dt}\Delta^{(m)}(t)\right) \star tx}{q-q^{-1}}. 
\end{equation*}
\end{lemma}
\begin{proof}
In the equation of Lemma \ref{lem:genfunrecursion2}, on both sides take the derivative with respect to $t$ and evaluate the result using the product rule. This yields 
\begin{align*}
\left(\frac{d}{dt}\Delta^{(m)}(t)\right)y^{-1}&=\frac{q^m x \star \Delta^{(m)}(t)-q^{-m} \Delta^{(m)}(t) \star x}{q-q^{-1}}\\
&\hspace{4em}+\frac{q^m tx \star \left(\frac{d}{dt}\Delta^{(m)}(t)\right)-q^{-m} \left(\frac{d}{dt}\Delta^{(m)}(t)\right) \star tx}{q-q^{-1}}. 
\end{align*}
Evaluate this equation using Lemma \ref{lem:genfunrecursion2} to obtain the result. 
\end{proof}

\begin{lemma}\label{lem:dD/dtlem2}\rm
For $m \in \mZ$, 
\begin{enumerate}
\item $\displaystyle \Big(\nabla^{(0)}(q^m t) \star \Delta^{(m)}(t)\Big)y^{-1}=q^m tx \star \Delta^{(m)}(t)\\
~\hspace{4em}+\frac{q^m}{q-q^{-1}}tx \star \nabla^{( 0)}(q^m t) \star \Delta^{(m)}(t)-\frac{q^{-m}}{q-q^{-1}}\nabla^{(0)}(q^m t) \star \Delta^{(m)}(t) \star tx;$
\item $\displaystyle \Big(\nabla^{(0)}(q^{-m} t) \star \Delta^{(m)}(t)\Big)y^{-1}=q^{-m} \Delta^{(m)}(t) \star tx\\
~\hspace{4em}+\frac{q^m}{q-q^{-1}}tx \star \nabla^{(0)}(q^{-m} t) \star \Delta^{(m)}(t)-\frac{q^{-m}}{q-q^{-1}}\nabla^{(0)}(q^{-m} t) \star \Delta^{(m)}(t) \star tx.$
\end{enumerate}
\end{lemma}
\begin{proof}
(i) Apply Lemma \ref{lem:y^-1} to the left-hand side and evaluate the result using Lemmas \ref{lem:genfunrecursion1} and \ref{lem:genfunrecursion2}. 

\medskip
\noindent (ii) Let $\Psi(t)$ denote the left-hand side minus the right-hand side. 

\medskip
\noindent Apply Lemma \ref{lem:y^-1} to $\Psi(t)$ and evaluate the result using Lemmas \ref{lem:genfunrecursion1} and \ref{lem:genfunrecursion2}. This yields 
\begin{align*}
&\Psi(t)=q^{-m}tx \star \Delta^{(m)}(t)-q^{-m}\Delta^{(m)}(t) \star tx\\
&\hspace{6em}+[m]_q\Big(\nabla^{(0)}(q^{-m}t) \star tx-tx \star \nabla^{(0)}(q^{-m}t)\Big) \star \Delta^{(m)}(t). 
\end{align*}

\noindent Evaluate the above big parenthesis using Lemma \ref{lem:genfunrecursion1}. This yields 
\begin{align*}
\Psi(t)&=q^{-m}tx \star \Delta^{(m)}(t)-q^{-m}\Delta^{(m)}(t) \star tx+(q^m-q^{-m})\Big(tx-q^m\nabla^{(0)}(q^{-m}t)y^{-1}\Big) \star \Delta^{(m)}(t)\\
&=q^mtx \star \Delta^{(m)}(t)-q^{-m}\Delta^{(m)}(t) \star tx-(q-q^{-1})q^m[m]_q\nabla^{(0)}(q^{-m}t)y^{-1} \star \Delta^{(m)}(t). 
\end{align*}

\noindent In the above result, evaluate the first two terms using Lemma \ref{lem:genfunrecursion2} and the third term using Proposition \ref{prop:D^m_-(t)}(i). This yields 
\begin{equation*}
\Psi(t)=(q-q^{-1})\Delta^{(m)}(t)y^{-1}-(q-q^{-1})\Delta^{(m)}(t)y^{-1}=0. 
\end{equation*}
\end{proof}

\begin{lemma}\label{lem:dD/dt}\rm
For $m \in \mZ$, 
\begin{equation*}
\frac{d}{dt}\Delta^{(m)}(t)=\frac{\nabla^{(0)}(q^m t)-\nabla^{(0)}(q^{-m}t)}{(q-q^{-1})t} \star \Delta^{(m)}(t). 
\end{equation*}
\end{lemma}
\begin{proof}
Define 
\begin{equation}\label{eq:defPhi}
\Phi(t)=\frac{d}{dt}\Delta^{(m)}(t)-\frac{\nabla^{(0)}(q^m t)-\nabla^{(0)}(q^{-m}t)}{(q-q^{-1})t} \star \Delta^{(m)}(t). 
\end{equation}
We will show $\Phi(t)=0$. 

\medskip
\noindent Write 
\begin{equation*}
\Phi(t)=\sum_{n \in \mN}\Phi_nt^n. 
\end{equation*}
We will show $\Phi_n=0$ for $n \in \mN$. 

\medskip
\noindent We claim that 
\begin{equation}\label{eq:Phiy^-1}
\Phi(t)y^{-1}=\frac{q^m tx \star \Phi(t)-q^{-m} \Phi(t) \star tx}{q-q^{-1}}. 
\end{equation}
To verify the claim, in \eqref{eq:Phiy^-1} eliminate $\Phi(t)$ everywhere using \eqref{eq:defPhi} and evaluate the result using Lemmas \ref{lem:genfunrecursion2}--\ref{lem:dD/dtlem2}. We have proved the claim. 

\medskip
\noindent We can now easily show $\Phi_n=0$ for $n \in \mN$. We do this by induction on $n$. 

\medskip
\noindent First we examine the constant term in \eqref{eq:defPhi}. Recall that $\Delta^{(m)}_0=\m1$, $\Delta^{(m)}_1=[m]_qxy$, $\nabla^{(0)}_1=xy$. Therefore, $\Phi_0=0$. 

\medskip
\noindent Next we show that $\Phi_n=0$ implies $\Phi_{n+1}=0$ for $n \in \mN$. 

\medskip
\noindent Let $n$ be given. We compare the coefficient of $t^{n+1}$ on both sides in \eqref{eq:Phiy^-1}. This yields 
\begin{equation*}
\Phi_{n+1}y^{-1}=\frac{q^m x \star \Phi_n-q^{-m} \Phi_n \star x}{q-q^{-1}}=0. 
\end{equation*}

\noindent In \eqref{eq:defPhi}, compare the coefficient of $t^{n+1}$ on both sides. This yields 
\begin{equation*}
\Phi_{n+1}=(n+2)\Delta^{(m)}_{n+2}-\sum_{k=0}^{n+1}[m(k+1)]_q\nabla^{(0)}_{k+1} \star \Delta^{(m)}_{n+1-k}. 
\end{equation*}

\noindent We evaluate the right-hand side of the above equation using Definitions \ref{def:D^m_n} and \ref{def:N^m_n} and expand out the result as a linear combination of words. Each word is contained in $\Cat_{n+2}$ by Lemma \ref{lem:catprod}. Each word ends with $y$ by Lemma \ref{lem:Catxy}. 

\medskip
\noindent By the above comment, 
\begin{equation*}
\Phi_{n+1}=\Phi_{n+1}y^{-1}y=0y=0. 
\end{equation*}

\noindent We have shown that $\Phi_n=0$ for $n \in \mN$, so $\Phi(t)=0$. The result follows. 
\end{proof}

\begin{theorem}\label{thm:D=exp}\rm
For $m \in \mZ$, 
\begin{equation*}
\Delta^{(m)}(t)=\exp\left(\sum_{n=1}^\infty\frac{[mn]_q}{n}\nabla^{(0)}_nt^n\right). 
\end{equation*}
In the above equation, the exponential power series is computed with respect to the $q$-shuffle product. 
\end{theorem}
\begin{proof}
Define 
\begin{equation*}
\Theta(t)=\exp\left(-\sum_{n=1}^\infty\frac{[mn]_q}{n}\nabla^{(0)}_nt^n\right). 
\end{equation*}

\noindent Then $\Theta(t)$ is invertible, with inverse 
\begin{equation*}
(\Theta(t))^{-1}=\exp\left(\sum_{n=1}^\infty\frac{[mn]_q}{n}\nabla^{(0)}_nt^n\right). 
\end{equation*}
We will show that $\Delta^{(m)}(t)$ is equal to the inverse of $\Theta(t)$. To do this, it suffices to show that $\Theta(t) \star \Delta^{(m)}(t)=\m1$. 

\medskip
\noindent By Lemma \ref{lem:N^0_ncommute}, the elements $\{\nabla^{(0)}_n\}_{n \geq 1}$ mutually commute. By the chain rule we have 
\begin{equation*}
\frac{d}{dt}\Theta(t)=\frac{\nabla^{(0)}(q^{-m} t)-\nabla^{(0)}(q^m t)}{(q-q^{-1})t} \star \Theta(t), 
\end{equation*}
\begin{equation}\label{eq:dTheta}
\frac{d}{dt}\Theta(t)=\Theta(t) \star \frac{\nabla^{(0)}(q^{-m} t)-\nabla^{(0)}(q^m t)}{(q-q^{-1})t}. 
\end{equation}

\noindent By the product rule and \eqref{eq:dTheta} we have 
\begin{align*}
\frac{d}{dt}\left(\Theta(t) \star \Delta^{(m)}(t)\right)&=\frac{d}{dt}\Theta(t) \star \Delta^{(m)}(t)+\Theta(t) \star \frac{d}{dt}\Delta^{(m)}(t)\\
&=\Theta(t) \star \frac{\nabla^{(0)}(q^{-m} t)-\nabla^{(0)}(q^m t)}{(q-q^{-1})t} \star \Delta^{(m)}(t)+\Theta(t) \star \frac{d}{dt}\Delta^{(m)}(t)\\
&=0, 
\end{align*}
where the last step follows by Lemma \ref{lem:dD/dt}. 

\medskip
\noindent Therefore, $\Theta(t) \star \Delta^{(m)}(t) \in \mV$. Since both $\Theta(t)$ and $\Delta^{(m)}(t)$ have constant term $\m1$, we have $\Theta(t) \star \Delta^{(m)}(t)=\m1$. The result follows. 
\end{proof}

\begin{corollary}\label{cor:comm}\rm
The elements in the set 
\begin{equation*}
\{\Delta^{(m)}_n\}_{m \in \mZ, n \in \mN} \cup \{\nabla^{(m)}_n\}_{m \in \mZ, n \geq 1} 
\end{equation*}
mutually commute with respect to the $q$-shuffle product. 
\end{corollary}
\begin{proof}
By Theorem \ref{thm:D=exp}, for $m \in \mZ$ and $n \geq 1$ the element $\Delta^{(m)}_n$ is a polynomial in the elements $\nabla^{(0)}_1, \nabla^{(0)}_2, \ldots, \nabla^{(0)}_n$. The result follows by Lemmas \ref{lem:D&N} and \ref{lem:N^0_ncommute}. 
\end{proof}

\begin{corollary}\label{cor:inverse}\rm
For $m \in \mZ$, 
\begin{equation*}
\Delta^{(-m)}(t) \star \Delta^{(m)}(t)=\m1=\Delta^{(m)}(t) \star \Delta^{(-m)}(t). 
\end{equation*}
\end{corollary}
\begin{proof}
Follows from Theorem \ref{thm:D=exp}. 
\end{proof}

\begin{corollary}\label{cor:D^1_n}\rm
For $n \in \mN$, 
\begin{equation*}
\Delta^{(1)}_n=(-1)^nD_n. 
\end{equation*}
\end{corollary}
\begin{proof}
Setting $m=1$ in Corollary \ref{cor:inverse} gives 
\begin{equation*}
\Delta^{(-1)}(t) \star \Delta^{(1)}(t)=\m1=\Delta^{(1)}(t) \star \Delta^{(-1)}(t). 
\end{equation*}
Recall that $D(t)$ is defined to be the inverse of $\tG(t)$ and that $\Delta^{(-1)}_n=(-1)^n\tG_n$ for $n \in \mN$. The result follows. 
\end{proof}

\noindent We now explain how the above results imply Theorems \ref{thm:genfuns}, \ref{thm:genfuns4}, \ref{thm:closedform}. In order to do this, we first recall the results of Lemmas \ref{lem:D^2_n}, \ref{lem:D^-1_n}, \ref{lem:D^0_n}, \ref{lem:N^0_n} and Corollary \ref{cor:D^1_n}. For $n \in \mN$, 
\begin{equation*}
\Delta^{(-1)}_n=(-1)^n\tG_n, \hspace{6em} \Delta^{(1)}_n=(-1)^nD_n, \hspace{6em} \Delta^{(2)}_n=C_n, 
\end{equation*}
\begin{equation*}
\Delta^{(0)}_n=\delta_{n,0}\m1, \hspace{6em}  \nabla^{(0)}_{n+1}=xC_ny. 
\end{equation*}

\noindent In terms of generating functions, we have 
\begin{equation*}
\Delta^{(-1)}(t)=\tG(-t), \hspace{6em} \Delta^{(1)}(t)=D(-t), \hspace{6em} \Delta^{(2)}(t)=C(t), 
\end{equation*}
\begin{equation*}
\Delta^{(0)}(t)=\m1, \hspace{6em} \nabla^{(0)}(t)=\sum_{n=1}^{\infty}xC_{n-1}yt^n. 
\end{equation*}

\begin{proof}[Proof of Theorem \ref{thm:genfuns}]
Follows by Theorem \ref{thm:D=exp} and Corollary \ref{cor:comm}. 
\end{proof}

\begin{proof}[Proof of Theorem \ref{thm:genfuns4}]
By direct computation we have that for $n \geq 1$, 
\begin{equation}\label{eq:qmn}
(q^n)^{m-1}+(q^n)^{m-3}+\cdots+(q^n)^{1-m}=\frac{q^{mn}-q^{-mn}}{q^n-q^{-n}}. 
\end{equation}
In order to show \eqref{eq:tG=exp}, eliminate each term on the left using Theorem \ref{thm:genfuns} and evaluate the result using \eqref{eq:qmn}. 

\medskip
\noindent We can show \eqref{eq:D=exp} using the same method. 
\end{proof}

\begin{proof}[Proof of Theorem \ref{thm:closedform}]
Follows by Theorems \ref{thm:genfuns4} and \ref{thm:D=exp}. 
\end{proof}

\section{Appendix A}
For the sake of completeness, we give some additional results involving the elements $\{\nabla^{(m)}_n\}_{m \in \mZ,n \geq 1}$ and $\{\Delta^{(m)}_n\}_{m \in \mZ,n \in \mN}$. 

\medskip
\noindent The following definition will be useful. 

\begin{definition}\label{def:zeta}\rm
(See \cite[Page 5]{ter_catalan}.) Let $\zeta:\mV \to \mV$ denote the $\mF$-linear map such that 
\begin{itemize}
\item $\zeta(x)=y$ and $\zeta(y)=x$; 
\item for a word $a_1 \cdots a_n$, 
\begin{equation*}
\zeta(a_1 \cdots a_n)=\zeta(a_n) \cdots \zeta(a_1). 
\end{equation*}
\end{itemize}
\end{definition}

\noindent By the above definition, the map $\zeta$ is an antiautomorphism on the free algebra $\mV$. Moreover, by the definition of the $q$-shuffle product in Section 2, the map $\zeta$ is an antiautomorphism on the $q$-shuffle algebra $\mV$. Thus for $v,w \in \mV$ we have 
\begin{equation*}
\zeta(vw)=\zeta(w)\zeta(v), \hspace{4em} \zeta(v \star w)=\zeta(w) \star \zeta(v). 
\end{equation*}

\begin{lemma}\label{lem:zetacat}\rm
For a Catalan word $w$ the following hold: 
\begin{enumerate}
\item $\zeta(w)$ is Catalan; 
\item assuming that $w$ is nontrivial, then $\nabla^{(m)}(w)=\nabla^{(m)}(\zeta(w))$ for $m \in \mZ$; 
\item $\Delta^{(m)}(w)=\Delta^{(m)}(\zeta(w))$ for $m \in \mZ$. 
\end{enumerate}
\end{lemma}
\begin{proof}
(i) Follows by Definitions \ref{def:Cat}. 

\medskip
\noindent (ii) Follows by Lemmas \ref{lem:N^m(w)} and \ref{lem:Ix=Iy}. 

\medskip
\noindent (iii) The case $w=\m1$ is routine, and the case $w \neq \m1$ follows from (ii) and Lemma \ref{lem:D&N(w)}. 
\end{proof}

\begin{lemma}\label{lem:zetainvariant}\rm
The map $\zeta$ fixes $\nabla^{(m)}_n$ for $m \in \mZ$ and $n \geq 1$. Moreover, $\zeta$ fixes $\Delta^{(m)}_n$ for $m \in \mZ$ and $n \in \mN$. 
\end{lemma}
\begin{proof}
Follows from Definitions \ref{def:D^m_n}, \ref{def:N^m_n} and Lemma \ref{lem:zetacat}. 
\end{proof}

\noindent Recall the $y^{-1}$ notation from Definition \ref{def:y^-1}. We now give a similar notation involving $x$. 

\begin{definition}\label{def:x^-1}\rm
(See \cite[Lemma 4.3]{PT}.) For $n \geq 1$ and a word $w=a_1a_2 \cdots a_n$, define 
\begin{equation*}
x^{-1}w=
\begin{cases}
a_2a_3 \cdots a_n,&\hspace{1em}\text{ if }a_1=x; \\
0,&\hspace{1em}\text{ if }a_1=y. 
\end{cases}
\end{equation*}
\noindent We further define $x^{-1}\m1=0$. For $v \in \mV$, we define $x^{-1}v$ in a linear way. 
\end{definition}

\begin{lemma}\label{lem:zetax-1y-1}\rm
For $v \in \mV$, 
\begin{equation*}
\zeta(vy^{-1})=x^{-1}\zeta(v). 
\end{equation*}
\end{lemma}
\begin{proof}
Follows from Definitions \ref{def:y^-1} and \ref{def:x^-1}. 
\end{proof}

\noindent Using the map $\zeta$ and Lemmas \ref{lem:zetainvariant}, \ref{lem:zetax-1y-1}, we obtain the following collection of identities. For completeness, we restate some identities proved earlier in the paper. 

\begin{proposition}\label{prop:altformulas1}\rm
For $m \in \mZ$ and $n \geq 1$, 
\begin{equation}\label{eq:N^m_nrecursionx}
\nabla^{(m)}_{n+1}=\frac{\left(q^m x \star \nabla^{(m)}_n-q^{-m} \nabla^{(m)}_n \star x\right)y}{q-q^{-1}}, 
\end{equation}
\begin{equation}\label{eq:N^m_nrecursiony}
\nabla^{(m)}_{n+1}=\frac{x\left(q^m \nabla^{(m)}_n \star y-q^{-m} y \star \nabla^{(m)}_n\right)}{q-q^{-1}}. 
\end{equation}
\end{proposition}
\begin{proof}
Equation \eqref{eq:N^m_nrecursionx} is from Proposition \ref{prop:N^m_nrecursion}. Equation \eqref{eq:N^m_nrecursiony} is obtained by applying $\zeta$ to \eqref{eq:N^m_nrecursionx}. 
\end{proof}

\begin{proposition}\label{prop:altformulas2}\rm
For $m \in \mZ$ and $n \in \mN$, 
\begin{equation}\label{eq:D^m_nrecursionx}
\Delta^{(m)}_{n+1}=\frac{\left(q^m x \star \Delta^{(m)}_n-q^{-m} \Delta^{(m)}_n \star x\right)y}{q-q^{-1}}, 
\end{equation}
\begin{equation}\label{eq:D^m_nrecursiony}
\Delta^{(m)}_{n+1}=\frac{x\left(q^m \Delta^{(m)}_n \star y-q^{-m} y \star \Delta^{(m)}_n\right)}{q-q^{-1}}. 
\end{equation}
\end{proposition}
\begin{proof}
Equation \eqref{eq:D^m_nrecursionx} is from Proposition \ref{prop:D^m_nrecursion}. Equation \eqref{eq:D^m_nrecursiony} is obtained by applying $\zeta$ to \eqref{eq:D^m_nrecursionx}. 
\end{proof}

\begin{proposition}\label{prop:altformulas3}\rm
For $m \in \mZ$, 
\begin{equation}\label{eq:D^m(t)1}
\Delta^{(m)}(t)y^{-1}=q^m[m]_q \nabla^{(0)}(q^{-m}t)y^{-1} \star \Delta^{(m)}(t), 
\end{equation}
\begin{equation}\label{eq:D^m(t)2}
\Delta^{(m)}(t)y^{-1}=q^{-m}[m]_q \Delta^{(m)}(t) \star \nabla^{(0)}(q^mt)y^{-1}, 
\end{equation}
\begin{equation}\label{eq:D^m(t)3}
x^{-1}\Delta^{(m)}(t)=q^m[m]_q \Delta^{(m)}(t) \star x^{-1}\nabla^{(0)}(q^{-m}t), 
\end{equation}
\begin{equation}\label{eq:D^m(t)4}
x^{-1}\Delta^{(m)}(t)=q^{-m}[m]_q x^{-1}\nabla^{(0)}(q^mt) \star \Delta^{(m)}(t). 
\end{equation}
\end{proposition}
\begin{proof}
Equations \eqref{eq:D^m(t)1}, \eqref{eq:D^m(t)2} are from Proposition \ref{prop:D^m_-(t)}. Equations \eqref{eq:D^m(t)3}, \eqref{eq:D^m(t)4} are obtained by applying $\zeta$ to \eqref{eq:D^m(t)1}, \eqref{eq:D^m(t)2} respectively. 
\end{proof}

\section{Acknowledgments}
The author is currently a Math Ph.D. student at the University of Wisconsin-Madison. The author would like to thank his supervisor, Professor Paul Terwilliger, for proofreading and giving many helpful comments.

\bigskip
\noindent Chenwei Ruan \\
Department of Mathematics \\
University of Wisconsin \\
480 Lincoln Drive \\
Madison, WI 53706-1388 USA \\
email: {\tt cruan4@wisc.edu }


\begin{thebibliography}{10}
\bibitem{ariki}
S.~Ariki. 
\newblock Representations of quantum algebras and combinatorics of Young tableaux.
\newblock {\em University Lecture Series},
26. Amer. Math. Soc., Providence, RI, 2002.


\bibitem{baseilhac2}
P.~Baseilhac.
\newblock The alternating presentation of $U_q(\widehat{gl_2})$ from Freidel-Maillet algebras.
\newblock{\em Nuclear Phys. B} 967 (2021) 115400; 
{\tt arXiv:2011.01572}.


\bibitem{beck}
J.~Beck. 
\newblock Braid group action and quantum affine algebras. 
\newblock{\em Commun. Math. Phys.} 165 (1994) 555--568; 
{\tt arXiv:hep-th/9404165}.

\bibitem{BCP}
J.~Beck, V.~Chari, A.~Pressley. 
\newblock An algebraic characterization of the affine canonical basis. 
\newblock{\em Duke Math. J.} (1999) 455--487; 
{\tt arXiv:math/9808060}.

\bibitem{BG}
M.~Bershtein, R.~Gonin.
\newblock Twisted and non-twisted deformed Virasoro algebras via vertex operators of $U_q(\widehat{\mathfrak{sl}}_2)$.
\newblock {\em Lett. Math. Phys.} 111 (2021) 22;
{\tt arXiv:2003.12472}.

\bibitem{bittmann}
L.~Bittmann.
\newblock Asymptotics of standard modules of quantum affine algebras.
\newblock {\em Algebr. Represent. Theory} 22 (2019) 1209--1237;
{\tt arXiv:1712.00355}.

\bibitem{brualdi}
R.~A.~Brualdi. 
\newblock{\em Introductory Combinatorics, Fifth Edition.}
\newblock Pearson, Upper Saddle River, NJ, 2010.


\bibitem{CP}
V.~Chari and A.~Pressley.
\newblock Quantum affine algebras.
\newblock {\em Commun. Math. Phys.}
142 (1991) 261--283. 

\bibitem{damiani}
I.~Damiani.
\newblock A basis of type Poincar\'e-Birkoff-Witt for the quantum algebra of $\widehat{\text{sl}}(2)$.
\newblock {\em J. Algebra} 161 (1993) 291--310.



\bibitem{drinfeld}
V.~G.~Drinfeld. 
\newblock Quantum groups.
\newblock{\em Proc. ICM Berkeley} 1 (1986) 789--820.

\bibitem{FHZ}
G.~Feng, N.~Hu, R. Zhuang. 
\newblock Another admissible quantum affine algebra of type $A_1^{(1)}$ with quantum Weyl group.
\newblock {\em J. Geom. Phys.} 165 (2021) 104218.


\bibitem{FR}
E.~Frenkel and N.~Reshetikhin.
\newblock
The $q$-characters of representations of quantum affine agebras and deformations of $W$-algebras, in:
\newblock {\em Recent developments in quantum affine algebras and related topics (Raleigh, NC, 1998)}, 
Amer. Math. Soc., Providence RI, 1999, pp. 163--205;
{\tt arXiv:math.QA/9810055}.

\bibitem{green}
J.~A.~Green.
\newblock Shuffle algebras, Lie algebras and quantum groups.
\newblock {\em Textos de Matem{\'a}tica. S{\'e}rie B [Texts in Mathematics. Series B]},
9. Universidade de Coimbra, Departamento de Matem{\'a}tica, Coimbra, 1995.

\bibitem{IT}
T.~Ito and P.~Terwilliger.
\newblock Tridiagonal pairs and the quantum affine algebra $U_q(\widehat{\mathfrak{sl}}_2)$.
\newblock{\em Ramanujan J.} 13 (2007) 39–62;
{\tt arXiv:math/0310042}.


\bibitem{jimbo}
M.~Jimbo. 
\newblock A $q$-difference analog of $U(g)$ and the Yang-Baxter equation.
\newblock{\em Lett. Math. Phys.} {10} (1985) 63--69.

\bibitem{JM}
M.~Jimbo, T.~Miwa.
\newblock Algebraic analysis of solvable lattice models.
\newblock {\em CBMS Regional Conference Series in Mathematics}, 
85. Amer. Math. Soc., Providence, RI, 1995.

\bibitem{jing}
N.~Jing.
\newblock Symmetric polynomials and $U_q(\widehat{\mathfrak{sl}}_2)$.
\newblock {\em Represent. Theory} 4 (2000) 46--63;
{\tt arXiv:math/9902109}.

\bibitem{JKKKY}
J.H.~Jung, S.-J.~Kang, M.~Kim, S.~Kim, J.-Y.~Yu.
\newblock Adjoint crystals and Young walls for $U_q(\widehat{\mathfrak{sl}}_2)$.
\newblock {\em European J. Combin.} 31 (2010) 738--758.

\bibitem{LIW}
X.~Liang, T.~Ito, Y.~Watanabe.
\newblock The Terwilliger algebra of the Grassmann scheme $J_q(N,D)$ revisited from the viewpoint of the quantum affine algebra $U_q(\widehat{\mathfrak{sl}}_2)$.
\newblock {\em Linear Algebra Appl.} 596 (2020) 117--144. 

\bibitem{lusztig}
G.~Lusztig.
\newblock {\em Introduction to quantum groups}.
\newblock Progress in Mathematics, 110. Birkhauser, Boston, 1993.

\bibitem{PT}
S.~Post and P.~Terwilliger.
\newblock An infinite-dimensional $\square_q$-module obtained from the $q$-shuffle algebra for affine $\mathfrak{sl}_2$.
\newblock{\em SIGMA Symmetry Integrability Geom. Methods Appl.} 
16 (2020) 037;
{\tt arXiv:1806.10007}.

\bibitem{rosso1}
M.~Rosso.
\newblock Groupes quantiques et alg{\`e}bres de battage quantiques.
\newblock{\em C.~R. Acad. Sci. Paris} 320 (1995) 145--148.

\bibitem{rosso2}
M.~Rosso.
\newblock Quantum groups and quantum shuffles.
\newblock{\em Invent. Math} 133 (1998) 399--416.

\bibitem{inverse}
C.~Ruan.
\newblock
A generating function associated with the alternating elements in the positive part of $U_q(\widehat{\mathfrak{sl}}_2)$. 
\newblock {\em Commun. Algebra}
51 (2023) 1707-1720;
{\tt arXiv:2204.10223}.





\bibitem{ter_alternating}
P.~Terwilliger.
\newblock The alternating PBW basis for the positive part of $U_q(\widehat{\mathfrak{sl}}_2)$.
\newblock {\em J. Math. Phys.} 60 (2019) 071704;
{\tt arXiv:1902.00721}.


\bibitem{ter_catalan}
P.~Terwilliger.
\newblock Using Catalan words and a $q$-shuffle algebra to describe a PBW basis for the positive part of $U_q(\widehat{\mathfrak{sl}}_2)$.
\newblock{\em J. Algebra} 525 (2019) 359--373;
{\tt arXiv:1806.11228}.

\bibitem{ter_beck}
P.~Terwilliger.
\newblock Using Catalan words and a $q$-shuffle algebra to describe the Beck PBW basis for the positive part of $U_q(\widehat{\mathfrak{sl}}_2)$.
\newblock{\em J. Algebra} 604 (2022) 162--184;
{\tt arXiv:2108.12708}.

\bibitem{watanabe}
Y.~Watanabe.
\newblock An algebra associated with a subspace lattice over a finite field and its relation to the quantum affine algebra $U_q(\widehat{\mathfrak{sl}}_2)$.
\newblock {\em J. Algebra} 489 (2017) 475--505. 

\bibitem{XXZ}
 J.~Xiao, H.~Xu, M.~Zhao. 
\newblock On bases of quantum affine algebras. 
\newblock {\em East China Normal University Scientific Reports} 16 (2022) 355-379; 
{\tt arXiv:2107.08631}.

\end{thebibliography}
\end{document}